\documentclass[11pt,reqno]{amsart}
\usepackage[a4paper,left=1.6cm,right=1.6cm,top=1.6cm,bottom=1.6cm]{geometry}
\usepackage{amsmath, amssymb, amsthm}
\usepackage{enumitem} 
\usepackage{tikz} 
\usetikzlibrary{arrows.meta}
\usepackage{multirow}
\usepackage{graphicx} 
\usepackage{caption} 
\usepackage{subcaption}
\usepackage{comment}
\usepackage{ytableau}
\usepackage[T1]{fontenc}
\usepackage{cite}
\usepackage{hyperref}
\usepackage{float}
\usepackage{colortbl}
\usepackage{xcolor}
\usepackage{tcolorbox}
\usepackage{appendix}
\usepackage{array}
\usepackage{pdflscape}
\usepackage{adjustbox} 
\usepackage{booktabs,longtable}
\usepackage{tabularx}
\usepackage{mathtools}
\usepackage{microtype}
\usepackage{array}
\usepackage{listings}
\usepackage{capt-of} 
\usepackage{lmodern}
\usepackage{anyfontsize}
\usepackage{tabularx}
\usepackage{seqsplit}
\usepackage{array}
\usepackage{algorithm}
\usepackage{algpseudocode} 
\newcolumntype{Y}{>{\centering\arraybackslash}X}
\newcolumntype{P}[1]{>{\raggedright\arraybackslash}p{#1}}
\newtheoremstyle{upright}%
{6pt}{6pt}
{\normalfont}
{}
{\bfseries}
{.}
{0.5em}
{}

\theoremstyle{upright}
\newtheorem{theorem}{Theorem}[section]
\newtheorem{lemma}[theorem]{Lemma}

\newtheorem{definition}[theorem]{Definition}

\theoremstyle{remark}
\newtheorem{remark}[theorem]{Remark}

\newcommand{\ch}{\mathrm{ch}}

\allowdisplaybreaks

\title[Rogers-Ramanujan type identities at $\Lambda_0$ from perfect crystals]{Rogers-Ramanujan type identities at $\Lambda_0$ from perfect crystals of exceptional quantum affine algebras}
\author[Shaolong Han]{Shaolong Han}
\address{Beijing International Center for Mathematical Research, Peking University, Beijing 100871, P.R.China}
\email{hanshaolong@bicmr.pku.edu.cn}

\begin{document}

\maketitle
\tableofcontents

\begin{abstract}
We derive Rogers--Ramanujan type partition identities at the fundamental weight $\Lambda_0$ for the exceptional affine types
$G_2^{(1)}$, $D_4^{(3)}$, $F_4^{(1)}$, $E_6^{(2)}$,  $E_6^{(1)}$, $E_7^{(1)}$ and $E_8^{(1)}$.
Our starting point is the Dousse--Konan reformulation of the $(\mathrm{KMN})^2$ crystal character formula, applied to the level-one perfect crystal $B=B(\theta)\sqcup B(0)$ of Benkart--Frenkel--Kang--Lee with ground element $\phi\in B(0)$.
This realizes the normalized character $e^{-\Lambda_0}\mathrm{ch} L(\Lambda_0)$ as generating functions of grounded $B$-colored partitions governed locally by the crystal energy.
After principal specialization, we obtain a colored partition model subject to explicit difference, congruence, and initial conditions.
On the product side, under the same specialization, the Weyl--Kac character formula yields an explicit Euler-type product, equivalently the generating function for partitions with parts in a concrete allowed set.
Comparing the two specializations gives coefficientwise equalities of generating functions.
A key computational feature is that the difference matrix can be produced from the crystal data
without explicitly computing the energy function.
For each type we tabulate the congruence data, forbidden initial parts, and the full difference
matrix, and we provide reproducible coefficient checks.
\end{abstract}

\section{Introduction}

Partition identities of Rogers--Ramanujan type form a classical meeting point of $q$-series,
additive combinatorics, and the representation theory of affine Kac--Moody algebras.  In their
combinatorial form, such identities assert that two seemingly unrelated generating functions
coincide: an \emph{Euler-type product} enumerating partitions with congruence restrictions on parts, and a (frequently colored) sum side enumerating partitions subject to difference and initial
conditions.  The paradigmatic examples are the Rogers--Ramanujan identities \cite{RR19,Sch17}, and
their many generalizations, including the Andrews--Gordon identities \cite{And74,Gor61} and
Bressoud's even-modulus analogues \cite{Bre79}.

\medskip

A representation-theoretic explanation of these phenomena originates in foundational work of
Lepowsky, Milne and Wilson \cite{LM78a,LM78b,LW84,LW85}.  For an affine Kac--Moody algebra, the
Weyl--Kac character formula \cite{Kac90,KP84} yields, after principal specialization, an explicit
product expansion.  This naturally supplies the product side of Rogers--Ramanujan type identities.
The essential challenge is the \emph{sum side}: the Weyl--Kac formula involves cancellations, so
its positivity is not manifest, and constructing sign-free bases (for instance via vertex operator
algebras) is delicate and strongly type-dependent.  This approach nevertheless led to many
remarkable identities, such as Capparelli's identities \cite{Cap93,Cap96}, the Meurman--Primc
identities \cite{MP87, MP99, MP01}, Siladi\'c's identity \cite{Sil17}, and Nandi's identities \cite{Nan14}.

\medskip

A conceptually uniform alternative is provided by crystal basis theory.  Kashiwara's existence
theorem \cite{Kas91} implies that integrable highest-weight characters admit sign-free expansions
over crystal graphs.  Building on this, Kang--Kashiwara--Misra--Miwa--Nakashima--Nakayashiki
introduced perfect crystals and established the $(\mathrm{KMN})^2$ crystal character formula
\cite{KMN92a,KMN92b,KKM94}, expressing irreducible characters as weighted sums over paths whose
local transitions are governed by an energy function on $B\otimes B$.  In this framework,
positivity is automatic and difference conditions arise \emph{locally} from the energy, avoiding
the need for global basis constructions.

\medskip

Primc \cite{Pr99} first exploited this mechanism at level one by comparing the principal
specializations of the Weyl--Kac and $(\mathrm{KMN})^2$ formulas, obtaining Rogers--Ramanujan type
identities in types $A_1^{(1)}$ and $A_2^{(1)}$.  Subsequent work clarified the combinatorial
content of the $(\mathrm{KMN})^2$ formula in the particularly tractable case of a constant ground
state path.  In this setting, Dousse and Konan reinterpreted the crystal character as the
generating function of \emph{grounded colored partitions} whose adjacent parts satisfy local,
energy-controlled difference conditions, leading to explicit sign-free character formulas and
sum/product identities in type $A_{n-1}^{(1)}$
\cite{DK20,DK22a}.  The framework was then extended beyond constant ground state paths via
\emph{multi-grounded partitions} \cite{DK22b}, producing further Rogers--Ramanujan type identities
in several classical affine types, including $A_{2n}^{(2)}$, $D_{n+1}^{(2)}$, $A_{2n-1}^{(2)}$,
$B_n^{(1)}$, and $D_n^{(1)}$, and more recently yielding higher-level companions of
Andrews--Gordon and Meurman--Primc identities in type $A_1^{(1)}$ \cite{DHK25}.  In parallel,
level-one perfect crystals in type $C_n^{(1)}$ have been used to obtain specialized and
non-specialized character formulas as generating functions of generalized colored partitions,
confirming level-one cases of the Capparelli--Meurman--Primc--Primc conjecture \cite{DK22c}.

\medskip

Despite these advances, exceptional affine types have remained comparatively less accessible from
this viewpoint.  While suitable level-one perfect crystals do exist---notably the uniform level-one perfect
crystals $B=B(\theta)\sqcup B(0)$ constructed by Benkart--Frenkel--Kang--Lee \cite{BFKL06}---one
still needs an effective way to extract the \emph{difference data} governing the sum side (and to
record it in a usable, type-uniform form). Recently, Dombos established a new Rogers--Ramanujan type identity using the perfect crystal theory of type $D_4^{(3)}$ \cite{Dom24}.

\medskip

The aim of the present paper is to carry out this program at the fundamental weight $\Lambda_0$ and to
provide fully explicit Rogers--Ramanujan type identities for all exceptional affine types $G_2^{(1)}$, $D_4^{(3)}$, $F_4^{(1)}$, $E_6^{(2)}$,  $E_6^{(1)}$, $E_7^{(1)}$ and $E_8^{(1)}$.
A key point is that the full difference matrix is computable purely from the crystal graph data, so the sum-side constraints become completely explicit.

\medskip

\noindent\textbf{Main result.}
For each exceptional affine type above, the principal specialization of the normalized
level-one character admits two manifestly positive descriptions:
\begin{itemize}
	\item a \emph{sum side} as the generating function of modified grounded $B$-colored partitions
	subject to three explicit families of constraints: (i) local difference conditions encoded by a
	matrix, (ii) congruence restrictions on part sizes determined by the colors, and (iii) finitely
	many forbidden initial colored parts;
	\item a \emph{product side} as an explicit Euler-type product, equivalently the generating function for
	partitions whose parts lie in a concrete allowed set.
\end{itemize}
Comparing these specializations yields Rogers--Ramanujan type partition identities in all seven
exceptional affine cases.

\medskip

\noindent\textbf{What is new.}
Beyond establishing the identities themselves, we make the sum-side models fully explicit and
computable:
\begin{enumerate}
	\item \emph{Type-uniform extraction of difference data.}
	We give an effective procedure to compute the shifted energy function $F$ on $B\otimes B$
	(and hence the full difference matrix $(F(b_1\otimes b_2))_{b_1,b_2\in B}$) directly from the
	crystal graph: starting from the normalization $F(\phi\otimes\phi)=0$, we traverse the connected graph
	$B\otimes B$ and propagate values along single $\widetilde f_i$-moves using the local recursion
	describing how $F$ changes under the action of Kashiwara operators.
	In particular, this avoids any need to
	explicitly compute the energy function.
	
	\item \emph{Complete explicit constraints.}
	For each exceptional affine type under consideration we tabulate the congruence data determining
	the product side, the finite list of forbidden initial colored parts, and the full difference
	matrix governing the sum-side difference conditions.  Together these yield a checkable colored
	partition model for the specialized character.
	
	\item \emph{Reproducible computations.}
	As an external verification, we expand both generating functions and confirm agreement of their
	$q$-expansions through degree $60$ in each type.
\end{enumerate}

\medskip

\noindent\textbf{Organization of the paper.}
Section~\ref{sec:qaa} recalls the definition of quantum affine algebras and fixes the notation
used throughout the paper.  Section~\ref{sec:pcanden} reviews level-one perfect crystals and their energy
functions, and introduces the shifted energy function $F$ on $B\otimes B$ which
govern the specialized difference conditions. 
Section~\ref{sec:exprecf} recalls the Dousse--Konan grounded-partition model for the
$(\mathrm{KMN})^{2}$ character formula and sets up the specialization framework, including the principal
specialization of the Weyl--Kac formula yielding the product side.

\medskip

Section~\ref{sec:los} shows that the principal specialization of $e^{-\Lambda_0}\ch L(\Lambda_0)$
admits an explicit Euler-type product expansion.
Section~\ref{sec:cidcon} derives the specialized
colored partition model and isolates the three families of sum-side constraints: the difference
conditions encoded by $F$, congruence restrictions determined by the colors, and finitely many
forbidden smallest colored parts.  Section~\ref{sec:pi} then introduces the product-side
congruence data, the multiplicities, and the set of allowed parts, and
proves the partition identity by identifying the two principal specializations and comparing
coefficients.  Finally, Section~\ref{sec:coeff-check-p60} records reproducible coefficient checks up to
$q^{60}$ for each exceptional affine type.  The appendices collect the explicit
congruence data and the difference matrices required to make the
identities completely explicit.

\medskip


\section{Quantum affine algebras}\label{sec:qaa}

Let $I=\{0,1,\dots,n\}$ be an index set, and let
$A=(a_{i,j})_{i,j\in I}$ be a Cartan matrix of affine type. Thus $A$ is
characterized by the following properties:
$a_{i,i}=2$ for all $i\in I$,
$a_{i,j}\in\mathbb Z_{\le 0}$ for $i\ne j$, and
$a_{i,j}=0$ if and only if $a_{j,i}=0$ for all distinct $i,j\in I$.
An affine Cartan matrix is always symmetrizable: there exists a diagonal matrix
$D=\mathrm{diag}(s_i)_{i\in I}$ with positive integers $s_i$ such that
$DA$ is symmetric.

The free abelian group $P^\vee
= \mathbb Z h_0 \oplus \mathbb Z h_1 \oplus \cdots \oplus \mathbb Z h_n
\oplus \mathbb Z d$
is called the \emph{dual weight lattice}. The linear functionals
$\alpha_i$ and $\Lambda_i$ ($i\in I$) on the complexification
$\mathfrak h:=\mathbb C\otimes_{\mathbb Z} P^\vee$ of $P^\vee$ are defined by
\begin{equation*}
	\begin{aligned}
		\langle h_j,\alpha_i\rangle &:= \alpha_i(h_j)=a_{j,i},
		&\qquad
		\langle d,\alpha_i\rangle &:= \alpha_i(d)=\delta_{i,0},\\
		\langle h_j,\Lambda_i\rangle &:= \Lambda_i(h_j)=\delta_{i,j},
		&\qquad
		\langle d,\Lambda_i\rangle &:= \Lambda_i(d)=0,
	\end{aligned}
	\qquad (i,j\in I).
\end{equation*}
The $\alpha_i$ are the \emph{simple roots} and the $\Lambda_i$ are the
\emph{fundamental weights}. Let $\Pi=\{\alpha_i \mid i\in I\}$ denote the set of
simple roots, and $\Pi^\vee=\{h_i \mid i\in I\}$ the set of simple coroots.

The weight lattice is $P=\{\lambda\in\mathfrak h^* \mid \lambda(P^\vee)\subset\mathbb Z\}$, and the set of \emph{dominant integral weights} is $P^+=\{\lambda\in P \mid \lambda(h_i)\in\mathbb Z_{\ge 0}\text{ for all } i\in I\}$.

Given $n\in\mathbb Z$ and an indeterminate $x$, set $[n]_x := \frac{x^n-x^{-n}}{x-x^{-1}}$.
Set $[0]_x!:=1$ and $[n]_x!:=[n]_x[n-1]_x\cdots[1]_x$ for $n\ge 1$, and for
$m\ge n\ge 0$ define the $x$-binomial coefficients by $\begin{bmatrix} m \\ n \end{bmatrix}_x
:= \frac{[m]_x!}{[n]_x!\,[m-n]_x!}$.

\begin{definition}\label{defn:qalg}
	The \emph{quantum affine algebra} $U_q(\widehat{\mathfrak g})$ associated with
	$(A,\Pi,\Pi^\vee,P,P^\vee)$ is the unital associative algebra over $\mathbb C(q)$
	(where $q$ is an indeterminate) generated by $e_i,f_i$ ($i\in I$) and $q^h$
	($h\in P^\vee$), subject to the defining relations
	\begin{itemize}
		\item[{\rm(i)}] $q^0=1,\quad q^h q^{h'}=q^{h+h'} \quad \text{for } h,h'\in P^\vee$;
		\item[{\rm(ii)}] $q^h e_i q^{-h}=q^{\alpha_i(h)} e_i \quad \text{for } h\in P^\vee,\ i\in I$;
		\item[{\rm(iii)}] $q^h f_i q^{-h}=q^{-\alpha_i(h)} f_i \quad \text{for } h\in P^\vee,\ i\in I$;
		\item[{\rm(iv)}] $e_i f_j-f_j e_i
		= \delta_{i,j}\,\dfrac{K_i-K_i^{-1}}{q_i-q_i^{-1}}
		\quad \text{for } i,j\in I$;
		\item[{\rm(v)}] $\displaystyle
		\sum_{k=0}^{1-a_{i,j}}
		\begin{bmatrix} 1-a_{i,j} \\ k \end{bmatrix}_{q_i}
		e_i^{\,1-a_{i,j}-k}\, e_j\, e_i^{\,k} = 0
		\quad \text{for } i\ne j \text{ in } I$;
		\item[{\rm(vi)}] $\displaystyle
		\sum_{k=0}^{1-a_{i,j}}
		\begin{bmatrix} 1-a_{i,j} \\ k \end{bmatrix}_{q_i}
		f_i^{\,1-a_{i,j}-k}\, f_j\, f_i^{\,k} = 0
		\quad \text{for } i\ne j \text{ in } I$,
	\end{itemize}
	where $q_i=q^{s_i}$ and $K_i=q^{s_i h_i}$.
\end{definition}

The \emph{canonical central element} $c$ and the \emph{null root} $\delta$ are
given by
\begin{equation*}
	c = c_0 h_0 + c_1 h_1 + \cdots + c_n h_n, \qquad
	\delta = d_0 \alpha_0 + d_1 \alpha_1 + \cdots + d_n \alpha_n,
\end{equation*}
where $c_0=1$, and $d_0=1$ except in type $A_{2n}^{(2)}$, where $d_0=2$.

We say that a dominant weight $\lambda\in P^+$ has \emph{level} $l$ if
$\langle c,\lambda\rangle:=\lambda(c)=l$.

The subalgebra $U'_q(\widehat{\mathfrak g})$ of $U_q(\widehat{\mathfrak g})$
generated by $e_i,f_i,K_i^{\pm1}$ ($i\in I$) is also commonly referred to as
the quantum affine algebra. The main difference between
$U_q(\widehat{\mathfrak g})$ and $U'_q(\widehat{\mathfrak g})$ is that
$U'_q(\widehat{\mathfrak g})$ admits nontrivial finite-dimensional irreducible
modules, whereas $U_q(\widehat{\mathfrak g})$ does not.

The elements of $\bar P=\bigoplus_{i\in I}\mathbb Z\,\Lambda_i$
are called \emph{classical weights}. Let
$\bar P^+:=\sum_{i=0}^n \mathbb Z_{\ge 0}\,\Lambda_i$
denote the corresponding set of dominant classical weights.

\section{Perfect crystals and energy functions}\label{sec:pcanden}

\begin{definition}\label{defn:pc}
	Let $l$ be a positive integer. A finite classical crystal $B$ is
	called a {\it perfect crystal of level $l$} for the quantum affine algebra
	$U_q(\widehat{\mathfrak g})$ if
	\begin{itemize}
		\item[{\rm(i)}] there exists a finite-dimensional $U'_q(\widehat{\mathfrak g})$-module
		with a crystal basis whose crystal graph is isomorphic to $B$;
		
		\item[{\rm(ii)}] the tensor product $B \otimes B$ is connected;
		
		\item[{\rm(iii)}] there exists a classical weight $\lambda_0$ such that $	\mathrm{wt}(B) \subset \lambda_0 + \frac{1}{d_0}\sum_{i \neq 0} \mathbb Z_{\leq 0}\,\alpha_i$ and $|B_{\lambda_0}| = 1$;

		\item[{\rm(iv)}] for any $b \in  B$, we have $\langle c,\varepsilon(b)\rangle=\sum_{i\in I}\varepsilon_i(b)\,\langle c,\Lambda_i\rangle\ge l$;
		
		\item[{\rm(v)}] for each $\lambda \in \bar P_l^+ := \{\mu \in \bar P^+ \mid \langle c,\mu\rangle =l\}$,
		there exist unique elements $b^\lambda,b_\lambda \in B$ such that
		$\varepsilon(b^\lambda) = \lambda$ and $\varphi(b_\lambda) = \lambda$.
	\end{itemize}
\end{definition}

In the sequel, we consider the level-one perfect crystals constructed by
Benkart--Frenkel--Kang--Lee \cite{BFKL06}, whose uniform construction is described as follows.

In the affine root system $\Phi$, removing the $0$-node of the affine Dynkin diagram yields
a finite type root system $\Phi^{0}$.  Write $\Phi^{0}_{+}$ for the set
of positive roots and $\Phi^{0}_{-} := -\Phi^{0}_{+}$ for the set of negative roots.  Let
$\Phi^{0}_{s,+}$, $\Phi^{0}_{l,+}$, $\Phi^{0}_{s,-} := -\Phi^{0}_{s,+}$ and
$\Phi^{0}_{l,-} := -\Phi^{0}_{l,+}$ denote, respectively, the sets of positive short roots,
positive long roots, negative short roots, and negative long roots.

We choose the data $(R^{+},R^{-},\Sigma,\theta)$
from the root system $\Phi^{0}$ as follows:
\begin{itemize}
	\item For  types $G_2^{(1)}$, $F_4^{(1)}$, $E_6^{(1)}$, $E_7^{(1)}$ and $E_8^{(1)}$, we set
	\[
	R^{+}=\Phi^{0}_{+},\quad R^{-}=\Phi^{0}_{-},\quad \Sigma=\{\text{simple roots of } \Phi^{0}_+\},\quad \theta\in\Phi^{0}_{+}\ \text{the highest root}.
	\]
	\item For  types $E_6^{(2)}$ and $D_4^{(3)}$, by an abuse of notation, we set
	\[
	R^{+}=\Phi^{0}_{s,+},\quad R^{-}=\Phi^{0}_{s,-},\quad \Sigma=\{\text{simple short roots of } \Phi^{0}_+\},\quad \theta\in\Phi^{0}_{s,+}\ \text{the highest short root}.
	\]
\end{itemize}

Define
\[
B(0):=\{\phi\},\qquad
B(\theta):=\{\,x_{\pm\alpha}\mid \alpha\in R^{+}\,\}\ \cup\ \{\,r_{\alpha_i}\mid \alpha_i\in\Sigma\,\},
\]
and set \(B:=B(\theta)\sqcup B(0)\).
We endow \(B\) with arrows (Kashiwara operators) as follows:
\begin{equation}\label{eq:crys}
	\begin{array}{cccc}
		\vspace{3mm}
		&(i\neq 0) &\qquad
		x_\alpha\ \stackrel{i}{\longrightarrow}\ x_\beta
		\ \Longleftrightarrow\ \alpha-\alpha_i=\beta
		&\qquad (\alpha,\beta\in R^{+}\cup R^{-})\\
		\vspace{3mm}
		& &\qquad
		x_{\alpha_i}\ \stackrel{i}{\longrightarrow}\ r_{\alpha_i}\ \stackrel{i}{\longrightarrow}\ x_{-\alpha_i}
		&\qquad (\alpha_i\in\Sigma)\\
		\vspace{3mm}
		&(i=0) &\qquad
		x_\alpha\ \stackrel{0}{\longrightarrow}\ x_\beta
		\ \Longleftrightarrow\ \alpha+\theta=\beta
		&\qquad (\alpha\in R^{-}\setminus\{-\theta\},\ \beta\in R^{+})\\
		\vspace{3mm}
		& &\qquad
		x_{-\theta}\ \stackrel{0}{\longrightarrow}\ \phi\ \stackrel{0}{\longrightarrow}\ x_{\theta}\,.
	\end{array}
\end{equation}
With this structure, \(B\) becomes a $U_q'(\widehat{\mathfrak g})$-crystal.

\begin{theorem}
[{\cite[Theorem 3.1]{BFKL06}}]
	The set $B = B(\theta) \sqcup B(0)$ with the structure given in \eqref{eq:crys}
	is a perfect crystal of level one for every quantum affine algebra
	$U'_q(\widehat{\mathfrak g})$.
\end{theorem}

\begin{definition}\label{def:energy}
	An \textit{energy function} on $B \otimes B$ is a map
	$H : B \otimes B \rightarrow \mathbb{Z}$ such that for all $i \in I$ and
	$b_1, b_2 \in B$ with $\widetilde{f}_i(b_1 \otimes b_2) \neq 0$, we have
	\begin{align*}
		H\bigl(\widetilde{f}_i(b_1 \otimes b_2)\bigr) =
		\begin{cases}
			H(b_1 \otimes b_2), & \text{if } i \neq 0, \\[2pt]
			H(b_1 \otimes b_2) - 1, & \text{if } i = 0 \text{ and } \varphi_0(b_1) > \varepsilon_0(b_2), \\[2pt]
			H(b_1 \otimes b_2) + 1, & \text{if } i = 0 \text{ and } \varphi_0(b_1) \leq \varepsilon_0(b_2).
		\end{cases}
	\end{align*}
\end{definition}

In the sequel, we assume the normalisation
\begin{equation}\label{eq:norm}
	H(\phi \otimes \phi) = 0.
\end{equation}

We define the height–shifted energy function
\begin{equation}\label{eq:shifted-energy-function}
F(b_1 \otimes b_2)
\ :=\
\operatorname{ht}(\delta)\,H(b_1 \otimes b_2)
\;+\;\operatorname{ht}(\overline{\mathrm{wt}}(b_1))
\;-\;\operatorname{ht}(\overline{\mathrm{wt}}(b_2)),
\end{equation}
where $\overline{\mathrm{wt}}(b)$ denotes the classical weight of $b \in B$, and
for a root $\alpha = \sum_j c_j \alpha_j$, we write
$\operatorname{ht}(\alpha) := \sum_j c_j$ for its height.

\vskip 2mm

\begin{lemma}\label{lem:Fgeq0}
	For all $b_1,b_2\in B$, we have $F(b_1\otimes b_2) \geq 0$.
\end{lemma}

\begin{proof}
	Set $h(b):= \operatorname{ht}(\overline{\mathrm{wt}}(b))$ for $b\in B$.
	By construction of $B=B(\theta)\sqcup B(0)$, we have $\overline{\mathrm{wt}}(\phi)=0$,
	$\overline{\mathrm{wt}}(r_{\alpha_i})=0$, and $\overline{\mathrm{wt}}(x_{\pm\alpha})=\pm\alpha$.
	Hence
	\begin{equation}\label{eq:hbound}
		-\operatorname{ht}(\theta)\ \le\ h(b)\ \le\ \operatorname{ht}(\theta)
		\qquad\text{for all }b\in B,
	\end{equation}
	and therefore
	\begin{equation}\label{eq:diffbound}
		h(b_1)-h(b_2)\ \ge\ -2\,\operatorname{ht}(\theta)
		\qquad\text{for all }b_1,b_2\in B.
	\end{equation}
	Also, since $\delta=\alpha_0+\theta$, we have
	\begin{equation}\label{eq:htdelta}
		\operatorname{ht}(\delta)=\operatorname{ht}(\theta)+1.
	\end{equation}
	
	Recall that the energy function $H$ is constant on each connected component of the
	(classical) crystal $B\otimes B$ obtained by deleting the $0$-arrows.
	Benkart--Frenkel--Kang--Lee classify these classical connected components and their
	energy values: $H$ takes values in $\{0,1,2\}$, with a unique component
	$C(x_\theta\otimes x_\theta)$ on which $H=2$, certain special components on which $H=0$,
	and all remaining components having $H=1$ (see \cite[Table 6.1]{BFKL06}).
	
	We now prove $F(b_1\otimes b_2)\ge 0$ by considering the value of $H$ on the classical
	connected component containing $b_1\otimes b_2$.
	
	\smallskip
	\noindent\emph{Case 1: $H(b_1\otimes b_2)=2$.}
	Then, using \eqref{eq:diffbound} and \eqref{eq:htdelta},
	\[
	F(b_1\otimes b_2)
	=2\,\operatorname{ht}(\delta)+h(b_1)-h(b_2)
	\ \ge\ 2(\operatorname{ht}(\theta)+1)-2\operatorname{ht}(\theta)
	=2>0.
	\]
	
	\smallskip
	\noindent\emph{Case 2: $H(b_1\otimes b_2)=0$.}
	By the classification, the $H=0$ components are exactly the singleton
	$C(\phi\otimes\phi)$, the singleton $C(x_\theta\otimes x_{-\theta})$, and (when they
	exist) the components $C(x_\theta\otimes r_{\alpha_i})\cong B(\theta)$
	(cf. \cite[Table 6.1, Proposition 5.1]{BFKL06}).
	
	\begin{itemize}
		\item If $b_1\otimes b_2=\phi\otimes\phi$, then $F(\phi\otimes\phi)=0$.
		\item If $b_1\otimes b_2=x_\theta\otimes x_{-\theta}$, then
		\[
		F(x_\theta\otimes x_{-\theta})=h(x_\theta)-h(x_{-\theta})
		=\operatorname{ht}(\theta)-(-\operatorname{ht}(\theta))=2\operatorname{ht}(\theta)>0.
		\]
		\item 
		If $b_1\otimes b_2\in C(x_\theta\otimes r_{\alpha_i})$, let
		\(\Psi: B(\theta)\xrightarrow{\sim} C(x_\theta\otimes r_{\alpha_i})\)
		be the crystal isomorphism of \cite[Proposition~5.1]{BFKL06}.
		By the explicit description of $\Psi$ in \cite[(5.4), (5.5), (5.8), (5.9)]{BFKL06},
		every element in the image $\Psi(B(\theta))=C(x_\theta\otimes r_{\alpha_i})$
		has first tensor factor either $r_{\alpha_j}$ or $x_\gamma$ with
		$\gamma\in R^{+}$, and second tensor factor either $r_{\alpha_j}$ or
		$x_{-\beta}$ with $\beta\in R^{+}$.
		Equivalently,
		\[
		\overline{\mathrm{wt}}(b_1)\in R^{+}\cup\{0\},
		\qquad
		\overline{\mathrm{wt}}(b_2)\in R^{-}\cup\{0\}.
		\]
		Hence $h(b_1)\ge 0\ge h(b_2)$ on this component, and therefore 
		\(
		F(b_1\otimes b_2)=h(b_1)-h(b_2)\ge 0.
		\)
	
	\end{itemize}
	
	Thus $F(b_1\otimes b_2)\ge 0$ whenever $H(b_1\otimes b_2)=0$.
	
	\smallskip
	\noindent\emph{Case 3: $H(b_1\otimes b_2)=1$.}
	We claim that in this case one cannot have simultaneously
	$h(b_1)<0$ and $h(b_2)>0$.
	Indeed, $h(b_1)<0$ forces $b_1=x_{-\gamma}$ for some $\gamma\in R^+$, and
	$h(b_2)>0$ forces $b_2=x_{\beta}$ for some $\beta\in R^+$.
	In the notation of \cite[Proposition~6.1]{BFKL06}, the partial order $\le$ on $R^+\cup R^-$ is
	defined by $\alpha\le\beta$ if and only if $\beta-\alpha\in\sum_{i\neq 0}\mathbb{Z}_{\ge 0}\alpha_i$.
	Since $\beta-(-\gamma)=\beta+\gamma$ is a nonnegative $\mathbb{Z}$-linear combination of
	simple roots, we have $-\gamma\le \beta$.
	Therefore \cite[Proposition~6.1]{BFKL06} implies
	\[
	x_{-\gamma}\otimes x_{\beta}\ \in\ C(x_\theta\otimes x_\theta),
	\]
	the unique component with $H=2$, contradicting the assumption $H=1$.
	
	Consequently, on every $H=1$ component we must have $h(b_1)\ge 0$ or $h(b_2)\le 0$.
	Using \eqref{eq:hbound}, in either subcase we obtain $h(b_1)-h(b_2)\ \ge\ -\operatorname{ht}(\theta)$.

	Hence, by \eqref{eq:htdelta},
	\[
	F(b_1\otimes b_2)
	=\operatorname{ht}(\delta)+h(b_1)-h(b_2)
	\ \ge\ (\operatorname{ht}(\theta)+1)-\operatorname{ht}(\theta)=1>0.
	\]
	
	\smallskip
	Combining the three cases shows that $F(b_1\otimes b_2)\ge 0$ for all $b_1,b_2\in B$.
\end{proof}

\section{Expressions for the character formulas}\label{sec:exprecf}

\begin{definition}
	Let $C$ be a set of colors, and define the set of \textit{colored integers}
	\[
	\mathcal{Z}_C \;=\; \{\,k^c \mid k\in\mathbb{Z},\ c\in C\,\}.
	\]
	Let $\succ$ be an \emph{binary relation} on $\mathcal{Z}_C$. A \textit{generalized colored partition} (with respect to $\succ$) is a finite sequence
	\[
	(\lambda^{c_1}_1,\ \lambda^{c_2}_2,\ \dots,\ \lambda^{c_s}_s)\quad \text{of elements of }\mathcal{Z}_C,
	\]
	such that for all $i\in\{1,\dots,s-1\}$ we have $\lambda_i \succ \lambda_{i+1}$.
\end{definition}

For a generalized colored partition $\lambda=(\lambda^{c_1}_1,\dots,\lambda^{c_s}_s)$, let
\[
\mathrm{pr}:\mathcal{Z}_C\to\mathbb{Z},\qquad \mathrm{pr}(k^c)=k.
\]
The \textit{weight} of $\lambda$ is $|\lambda| \;=\; \sum_{i=1}^s \mathrm{pr}(\lambda^{c_i}_i)$, 
and the number of \textit{nonzero} parts is $\ell(\lambda) \;:=\; \#\{\,1\le i\le s \mid \mathrm{pr}(\lambda^{c_i}_i)\neq 0\,\}$.

\vskip 2mm

Fix a color $g\in C$. A \emph{grounded partition} with ground $g$ (and relation $\succ$)
is a nonempty generalized colored partition $\lambda=(\lambda^{c_1}_1,\dots,\lambda^{c_s}_s)\qquad (s\ge 1)$
such that $\lambda_s^{c_s}=0^g$ and, when $s\ge 2$, we have $\lambda_{s-1}\ne 0^g$.
Let $\mathcal{P}_g^\succ$ denote the set of all grounded partitions with ground $g$
and relation $\succ$.

\medskip

Let $B$ be the perfect crystal given in \eqref{eq:crys}. By \cite[Theorem~6.2]{BFKL06},
the energy function $H$ on $B\otimes B$ takes values in $\{0,1,2\}$.
We use $H$ to define a binary relation on colored integers as follows.
Take the color set to be $C=B$, and define a relation $\gg$ on $\mathcal Z_B$ by
\begin{equation}\label{eq:energy-condition}
	k^{\,b}\gg (k')^{\,b'}
	\ \Longleftrightarrow\
	k-k'\ge H(b'\otimes b),
\end{equation}
for $b,b'\in B$ and $k,k'\in\mathbb Z$.
Thus the admissibility condition~\eqref{eq:energy-condition} is completely determined by
the values of the energy function. The energy matrices for types
$D_4^{(3)}$, $G_2^{(1)}$, $E_6^{(2)}$ and $F_4^{(1)}$ are computed explicitly in
\cite{FHKS24,HJKL24}.

For any generalized colored partition
$\lambda=(\lambda^{b_1}_1,\dots,\lambda^{b_s}_s)$ with colors $b_1,\dots,b_s\in B$, define $C(\lambda):=e^{\sum_{k=1}^s\overline{\mathrm{wt}}(b_k)}$.

For $a\in\mathbb C[q]$, we use the $q$-Pochhammer symbol $(a;q)_\infty := \prod_{k=0}^{\infty} (1-aq^k)$.
For convenience, we also write $(x;y):= (q^x;q^y)_\infty$.

We then have the following expression of character formula.

\begin{theorem}[{\cite[Theorem~3.8]{DK20}}]\label{thm:coloredKMN}
	Let $l\in\mathbb Z_{>0}$ and let $B$ be a perfect crystal of level $l$ for
	$U_q(\widehat{\mathfrak g})$. Let $\mu\in \bar P_l^+$, and let $L(\mu)$ be the
	irreducible highest-weight $U_q(\widehat{\mathfrak g})$-module of highest weight $\mu$.
	Assume that the ground-state path is constant $p_\mu=\cdots\otimes g\otimes g\otimes g$
	for some $g\in B$. Set $q:=e^{-\delta/d_0}$. Then
	\begin{equation}\label{eq:coloredKMN}
		\sum_{\lambda\in\mathcal P_g^\gg} C(\lambda)\,q^{|\lambda|}
		= \frac{e^{-\mu}\,\mathrm{ch}\,L(\mu)}{(q;q)_\infty}.
	\end{equation}
\end{theorem}

We will only need the level-one case $\mu=\Lambda_0$ in the exceptional affine types
considered in this paper. By Definition~\ref{defn:pc} and the structure
\eqref{eq:crys}, one has $b_{\Lambda_0}=b^{\Lambda_0}=\phi$, hence the ground-state path
is constant $p_{\Lambda_0}=\cdots\otimes \phi\otimes \phi\otimes \phi$.
Accordingly, we write $\mathcal P_\phi^\gg$ for the set of grounded partitions with
ground $\phi$ satisfying the relation $\gg$ in \eqref{eq:energy-condition}.

\begin{lemma}\label{lem:lambda-penultimate-positive}
	Let $\lambda=(\lambda^{b_1}_1,\dots,\lambda^{b_s}_s)\in\mathcal P_\phi^\gg$. Then
	$\lambda_{s-1}\ge 1$.
\end{lemma}
\begin{proof}
	If $\lambda_{s-1}=0$, then $b_{s-1}\ne\phi$ by groundedness of $\lambda$.
	Applying \eqref{eq:energy-condition} to the last step gives
	\[
	0=\lambda_{s-1}-\lambda_s \ge H(b_s\otimes b_{s-1})
	=H(\phi\otimes b_{s-1}).
	\]
	By \cite[Table~6.1, (6.2)]{BFKL06}, one has $H(\phi\otimes b)=1$ for all $b\in B(\theta)$,
	a contradiction. Hence $\lambda_{s-1}\ne 0$, i.e.\ $\lambda_{s-1}\ge 1$.
\end{proof}

The specialization of the character formula for  $L(\mu)$ can
also be determined from the structure of the root system, as follows.
For a sequence $\mathbf s=(s_0,s_1,\ldots,s_n)$ of nonnegative integers, define the
$\mathbf s$-\emph{specialization}
\[
\mathrm{F}_{\mathbf s}\colon
\mathbb Z[[e^{-\alpha_0},e^{-\alpha_1},\ldots,e^{-\alpha_n}]]
\to \mathbb Z[[q]],
\qquad
e^{-\alpha_i}\mapsto q^{s_i}.
\]
The special case $\mathbf s=(1,1,\ldots,1)$ is called the \emph{principal specialization}
and is denoted by $\mathrm{F}_1$. For a root system $\Phi$, consider
\[
D(\Phi):=\prod_{\alpha\in\Phi^+}\left(1-e^{-\alpha}\right)^{\dim \mathfrak g_{-\alpha}}.
\]
Then the principal specialization of the normalized character has the following product
expansion.

\begin{theorem}
[{\cite[Proposition~10.9]{Kac90}}]	
\label{thm:Lepowsky}
	\[
	\mathrm{F}_1\!\left(e^{-\mu}\,\mathrm{ch}\,L(\mu)\right)
	=
	\frac{\mathrm{F}_{(\mu(h_0)+1,\ldots,\mu(h_n)+1)}\, D(\Phi^\vee)}
	{\mathrm{F}_1\, D(\Phi^\vee)}.
	\]
\end{theorem}

\section{The level-one specialization}\label{sec:los}

Throughout this section, we collect the root-theoretic input needed to apply
Theorem~\ref{thm:Lepowsky} to the level-one weight $\Lambda_0$.
Let $\{\alpha_1,\dots,\alpha_n\}$ be the simple roots of the finite root system
$\Phi^{0}$ obtained by removing the $0$-node from the affine Dynkin diagram.
For $(a_1,\dots,a_n)\in\mathbb Z_{\ge 0}^n$, set
\[
(a_1\cdots a_n):=\sum_{i=1}^n a_i\,\alpha_i,
\qquad
\overline{(a_1\cdots a_n)}:=-\sum_{i=1}^n a_i\,\alpha_i,
\]
and, for affine roots, similarly set $(a_0a_1\cdots a_n):=\sum_{i=0}^n a_i\,\alpha_i$.

With this convention, Tables~\ref{tab:stacked-positive-roots} and~\ref{tab:affinepositive}
list the finite positive roots and the corresponding descriptions of $\Phi_+$
(including $\delta$ and the multiplicities of imaginary roots) for all exceptional
affine types considered in this paper.


\newcolumntype{L}[1]{>{\raggedright\arraybackslash}m{#1}}

\setlength{\tabcolsep}{5pt}
\renewcommand{\arraystretch}{1.12}

\begin{longtable}{|p{.05\textwidth}|L{.32\textwidth}|L{.58\textwidth}|}
	\caption{Finite positive roots}
	\label{tab:stacked-positive-roots}\\
	\hline
	Type & $\Phi^{0}_{l,+}$ (Positive long roots) & $\Phi^{0}_{s,+}$ (Positive short roots) \\ \hline
	\endfirsthead
	\hline
	Type & \multicolumn{2}{|L{.92\textwidth}}{$\Phi^{0}_{+}$ (All positive roots)}\\ \hline
	\endhead
	\hline
	\endfoot
	
	$G_2^{(1)}$ &
	{\footnotesize\begin{tabular}[t]{@{}l@{}}(10)\ \ (13)\ \ (23)\end{tabular}} &
	{\footnotesize\begin{tabular}[t]{@{}l@{}}(01)\ \ (11)\ \ (12)\end{tabular}} \\
	\hline
	
	$D_4^{(3)}$ &
	{\footnotesize\begin{tabular}[t]{@{}l@{}}(01)\ \ (31)\ \ (32)\end{tabular}} &
	{\footnotesize\begin{tabular}[t]{@{}l@{}}(10)\ \ (11)\ \ (21)\end{tabular}} \\
	\hline
	
	$F_4^{(1)}$ &
	{\footnotesize\begin{tabular}[t]{@{}l@{}}
			(1000)\ \ (0100)\ \ (1100)\ \ (0120) \\
			(1120)\ \ (1220)\ \ (0122)\ \ (1122) \\
			(1222)\ \ (1242)\ \ (1342)\ \ (2342)
	\end{tabular}} &
	{\footnotesize\begin{tabular}[t]{@{}l@{}}
			(0010)\ \ (0110)\ \ (1110)\ \ (1232) \\
			(0001)\ \ (0011)\ \ (0111)\ \ (0121) \\
			(1111)\ \ (1121)\ \ (1221)\ \ (1231)
	\end{tabular}} \\
	\hline
	
	$E_6^{(2)}$ &
	{\footnotesize\begin{tabular}[t]{@{}l@{}}
			(0001)\ \ (0010)\ \ (0011)\ \ (0210) \\
			(0211)\ \ (0221)\ \ (2210)\ \ (2211) \\
			(2221)\ \ (2421)\ \ (2431)\ \ (2432)
	\end{tabular}} &
	{\footnotesize\begin{tabular}[t]{@{}l@{}}
			(0100)\ \ (0110)\ \ (0111)\ \ (2321) \\
			(1000)\ \ (1100)\ \ (1110)\ \ (1210) \\
			(1111)\ \ (1211)\ \ (1221)\ \ (1321)
	\end{tabular}} \\
	\hline
	
	Type & \multicolumn{2}{L{.92\textwidth}|}{$\Phi^{0}_{+}$ (All positive roots)}\\ \hline
	
	$E_6^{(1)}$ & \multicolumn{2}{L{.92\textwidth}|}{
		\footnotesize\begin{tabular}{@{}*{8}{l}@{}}
		(100000) & (010000) & (001000) & (000100) & (000010) & (000001) & (101000) & (010100) \\
		(001100) & (000110) & (000011) & (101100) & (011100) & (010110) & (001110) & (000111) \\
		(111100) & (101110) & (011110) & (010111) & (001111) & (111110) & (101111) & (011210) \\
		(011111) & (111210) & (111111) & (011211) & (112210) & (111211) & (011221) & (112211) \\
		(111221) & (112221) & (112321) & (122321) \\
		\end{tabular}
	}\\ \hline
	
	$E_7^{(1)}$ & \multicolumn{2}{L{.92\textwidth}|}{
		\footnotesize\begin{tabular}{@{}*{8}{l}@{}}
	(1000000) & (0100000) & (0010000) & (0001000) & (0000100) & (0000010) & (0000001) & (1010000) \\
	(0101000) & (0011000) & (0001100) & (0000110) & (0000011) & (1011000) & (0111000) & (0101100) \\
	(0011100) & (0001110) & (0000111) & (1111000) & (1011100) & (0111100) & (0101110) & (0011110) \\
	(0001111) & (1111100) & (1011110) & (0112100) & (0111110) & (0101111) & (0011111) & (1112100) \\
	(1111110) & (1011111) & (0112110) & (0111111) & (1122100) & (1112110) & (1111111) & (0112210) \\
	(0112111) & (1122110) & (1112210) & (1112111) & (0112211) & (1122210) & (1122111) & (1112211) \\
	(0112221) & (1123210) & (1122211) & (1112221) & (1223210) & (1123211) & (1122221) & (1223211) \\
	(1123221) & (1223221) & (1123321) & (1223321) & (1224321) & (1234321) & (2234321) \\
		\end{tabular}
	}\\ \hline
	
	$E_8^{(1)}$ & \multicolumn{2}{L{.92\textwidth}|}{
		\footnotesize\begin{tabular}{@{}*{8}{l}@{}}
		(10000000) & (01000000) & (00100000) & (00010000) & (00001000) & (00000100) & (00000010) & (00000001) \\
		(10100000) & (01010000) & (00110000) & (00011000) & (00001100) & (00000110) & (00000011) & (10110000) \\
		(01110000) & (01011000) & (00111000) & (00011100) & (00001110) & (00000111) & (11110000) & (10111000) \\
		(01111000) & (01011100) & (00111100) & (00011110) & (00001111) & (11111000) & (10111100) & (01121000) \\
		(01111100) & (01011110) & (00111110) & (00011111) & (11121000) & (11111100) & (10111110) & (01121100) \\
		(01111110) & (01011111) & (00111111) & (11221000) & (11121100) & (11111110) & (10111111) & (01122100) \\
		(01121110) & (01111111) & (11221100) & (11122100) & (11121110) & (11111111) & (01122110) & (01121111) \\
		(11222100) & (11221110) & (11122110) & (11121111) & (01122210) & (01122111) & (11232100) & (11222110) \\
		(11221111) & (11122210) & (11122111) & (01122211) & (12232100) & (11232110) & (11222210) & (11222111) \\
		(11122211) & (01122221) & (12232110) & (11232210) & (11232111) & (11222211) & (11122221) & (12232210) \\
		(12232111) & (11233210) & (11232211) & (11222221) & (12233210) & (12232211) & (11233211) & (11232221) \\
		(12243210) & (12233211) & (12232221) & (11233221) & (12343210) & (12243211) & (12233221) & (11233321) \\
		(22343210) & (12343211) & (12243221) & (12233321) & (22343211) & (12343221) & (12243321) & (22343221) \\
		(12343321) & (12244321) & (22343321) & (12344321) & (22344321) & (12354321) & (22354321) & (13354321) \\
		(23354321) & (22454321) & (23454321) & (23464321) & (23465321) & (23465421) & (23465431) & (23465432) \\	
		\end{tabular}
	}\\ \hline
	
\end{longtable}

For types $D_4^{(3)}$ and $E_6^{(2)}$, the set of positive roots is 
\begin{align*}
	\Phi_{+}=\Phi_{+}^{re}\cup \Phi_+^{im}=&\{\alpha+r\delta\mid \alpha\in \Phi_{s,+}^{0},\ r\geq 0\ \text{or}\ \alpha\in\Phi_{s,-}^{0},\ r>0\}\\
	&\cup\{\alpha+rk\delta\mid \alpha\in \Phi_{l,+}^{0},\ r\geq 0\ \text{or}\ \alpha\in\Phi_{l,-}^{0},\ r>0\} \cup\{r\delta\mid r>0\},
\end{align*}
where $k=3$ for $D_4^{(3)}$ and $k=2$ for $E_6^{(2)}$.

\vskip 4mm

For types $G_2^{(1)}$, $F_4^{(1)}$, $E_6^{(1)}$, $E_7^{(1)}$ and $E_8^{(1)}$, the set of positive roots is 
\begin{align*}
	\Phi_{+}=\Phi_{+}^{re}\cup \Phi_+^{im}
	&=\{\alpha+r\delta\mid \alpha\in \Phi_{+}^0,\ r\geq 0 \ \text{or}\ \alpha\in \Phi_{-}^0,\ r\geq 1\} 
	\cup \{r\delta\mid r>0\}.
\end{align*}

More precisely, the set $\Phi_+$ of each type is given by the following table

\setlength{\tabcolsep}{4pt}
\renewcommand{\arraystretch}{1.12}

\begin{longtable}{|l|P{0.11\textwidth}|P{0.45\textwidth}|P{0.3\textwidth}@{}|}
\caption{Null roots and affine positive roots}\label{tab:affinepositive}\\		
	\hline
	Type & Null root $\delta$ & Real roots & Imaginary roots  \\
	\hline
	\endfirsthead
	\hline
	Type & Null root $\delta$ & Real roots & Imaginary roots  \\
	\hline
	\endhead
	\hline
	\endfoot

$G_2^{(1)}$ &
\((123)\) &
\begin{tabular}[t]{@{}l@{}}
	\(\{\alpha+r\delta \mid \alpha\in \Phi^{0}_{+},\ r\ge 0\ \text{or}\ \alpha\in \Phi^{0}_{-},\ r\ge 1\}\)
\end{tabular}
&
\begin{tabular}[t]{@{}l@{}}
	\(\{k\delta\mid k\ge 1\}\) (mult.\ 2)
\end{tabular}
\\ \hline

	$D_4^{(3)}$ &
	\((121)\) &
	\begin{tabular}[t]{@{}l@{}}
		\(\{\alpha+r\delta \mid \alpha\in \Phi^{0}_{s,+},\ r\ge 0\ \text{or}\ \alpha\in \Phi^{0}_{s,-},\ r\ge 1\}\)\hfill\(\cup\)\\
		\(\{\alpha+3r\delta \mid \alpha\in \Phi^{0}_{l,+},\ r\ge 0\ \text{or}\ \alpha\in \Phi^{0}_{l,-},\ r\ge 1\}\)
	\end{tabular}
	&
	\begin{tabular}[t]{@{}l@{}}
		\(\{3k\delta\mid k\ge 1\}\) (mult.\ 2)\ \(\cup\)\\
		\(\{(3k+1)\delta\mid k\ge 0\}\) (mult.\ 1)\ \(\cup\)\\
		\(\{(3k+2)\delta\mid k\ge 0\}\) (mult.\ 1)
	\end{tabular}
	\\ \hline
	
	$E_6^{(2)}$ &
	\((12321)\) &
	\begin{tabular}[t]{@{}l@{}}
		\(\{\alpha+r\delta \mid \alpha\in \Phi^{0}_{s,+},\ r\ge 0\ \text{or}\ \alpha\in \Phi^{0}_{s,-},\ r\ge 1\}\)\hfill\(\cup\)\\
		\(\{\alpha+2r\delta \mid \alpha\in \Phi^{0}_{l,+},\ r\ge 0\ \text{or}\ \alpha\in \Phi^{0}_{l,-},\ r\ge 1\}\)
	\end{tabular}
	&
	\begin{tabular}[t]{@{}l@{}}
		\(\{2k\delta\mid k\ge 1\}\) (mult.\ 4)\ \(\cup\)\\
		\(\{(2k+1)\delta\mid k\ge 0\}\) (mult.\ 2)
	\end{tabular}
	\\ \hline

	$F_4^{(1)}$ &
	\((12342)\) &
	\begin{tabular}[t]{@{}l@{}}
		\(\{\alpha+r\delta \mid \alpha\in \Phi^{0}_{+},\ r\ge 0\ \text{or}\ \alpha\in \Phi^{0}_{-},\ r\ge 1\}\)
	\end{tabular}
	&
	\begin{tabular}[t]{@{}l@{}}
		\(\{k\delta\mid k\ge 1\}\) (mult.\ 4)
	\end{tabular}
	\\ \hline
	
	$E_6^{(1)}$ &
	\((1122321)\) &
	\begin{tabular}[t]{@{}l@{}}
		\(\{\alpha+r\delta \mid \alpha\in \Phi^{0}_{+},\ r\ge 0\ \text{or}\ \alpha\in \Phi^{0}_{-},\ r\ge 1\}\)
	\end{tabular}
	&
	\begin{tabular}[t]{@{}l@{}}
		\(\{k\delta\mid k\geq 1\}\) (mult.\ 6)
	\end{tabular}
	\\ \hline
	
	$E_7^{(1)}$ &
	\((12234321)\) &
	\begin{tabular}[t]{@{}l@{}}
		\(\{\alpha+r\delta \mid \alpha\in \Phi^{0}_{+},\ r\ge 0\ \text{or}\ \alpha\in \Phi^{0}_{-},\ r\ge 1\}\)
	\end{tabular}
	&
	\begin{tabular}[t]{@{}l@{}}
		\(\{k\delta\mid k\geq 1\}\) (mult.\ 7)
	\end{tabular}
	\\ \hline
	
	$E_8^{(1)}$ &
	\((123465432)\) &
	\begin{tabular}[t]{@{}l@{}}
		\(\{\alpha+r\delta \mid \alpha\in \Phi^{0}_{+},\ r\ge 0\ \text{or}\ \alpha\in \Phi^{0}_{-},\ r\ge 1\}\)
	\end{tabular}
	&
	\begin{tabular}[t]{@{}l@{}}
		\(\{k\delta\mid k\geq 1\}\) (mult.\ 8)
	\end{tabular}
	\\ \hline
	
\end{longtable}

For any $\beta=\sum_{i=0}^{n}m_i\alpha_i$ and $\mathbf s=(s_0,s_1,\cdots,s_n)$, we denote $\mathrm{ht}_{\mathbf s}(\beta)=\sum_{i=0}^ns_im_i$. Then we have $\mathrm{ht}=\mathrm{ht}_{(1,\cdots,1)}$.

The Langlands duals of $G_2^{(1)}$ and $F_4^{(1)}$ are $D_4^{(3)}$ and $E_6^{(2)}$, respectively. Then the specialization $\mathrm{F}_{\mathbf s}D(\Phi^\vee)$ is given by
\begin{align*}
	&\mathrm F_{\mathbf s}\!\left(\prod_{\alpha\in\Phi^{+}}
	\bigl(1-e^{-\alpha}\bigr)^{\dim\mathfrak g_{-\alpha}}\right)\\
	=&\prod_{\alpha\in\Phi_{s,+}^0}\prod_{r\ge 0}\!\!\left(1-q^{\operatorname{ht}_{\mathbf s}(\alpha)+r \operatorname{ht}_{\mathbf s}(\delta)}\right)
	\prod_{\alpha\in\Phi_{s,+}^0}\prod_{r\ge 1}\!\!\left(1-q^{-\operatorname{ht}_{\mathbf s}(\alpha)+r \operatorname{ht}_{\mathbf s}(\delta)}\right)\\
	&\times
	\prod_{\alpha\in\Phi_{l,+}^0}\prod_{r\ge 0}\!\!\left(1-q^{\operatorname{ht}_{\mathbf s}(\alpha)+rk\operatorname{ht}_{\mathbf s}(\delta)}\right)
	\prod_{\alpha\in\Phi_{l,+}^0}\prod_{r\ge 1}\!\!\left(1-q^{-\operatorname{ht}_{\mathbf s}(\alpha)+rk\operatorname{ht}_{\mathbf s}(\delta)}\right) 
	\prod_{\alpha\in\Phi^{im}_{+}}\!\!\left(1-q^{\operatorname{ht}_{\mathbf s}(\alpha)}\right)^{{\rm mult}(\alpha)}
	\\
	=&\prod_{\alpha\in\Phi_{s,+}^0}\prod_{r\ge 0}\!\!\left(1-q^{\operatorname{ht}_{\mathbf s}(\alpha)+r \operatorname{ht}_{\mathbf s}(\delta)}\right)
	\prod_{\alpha\in\Phi_{s,+}^0}\prod_{r\ge 0}\!\!\left(1-q^{\operatorname{ht}_{\mathbf s}(\delta)-\operatorname{ht}_{\mathbf s}(\alpha)+r \operatorname{ht}_{\mathbf s}(\delta)}\right)\\
	&\times
	\prod_{\alpha\in\Phi_{l,+}^0}\prod_{r\ge 0}\!\!\left(1-q^{\operatorname{ht}_{\mathbf s}(\alpha)+rk\operatorname{ht}_{\mathbf s}(\delta)}\right)
	\prod_{\alpha\in\Phi_{l,+}^0}\prod_{r\ge 0}\!\!\left(1-q^{k\operatorname{ht}_{\mathbf s}(\delta)-\operatorname{ht}_{\mathbf s}(\alpha)+rk\operatorname{ht}_{\mathbf s}(\delta)}\right) 
	\prod_{\alpha\in\Phi^{im}_{+}}\!\!\left(1-q^{\operatorname{ht}_{\mathbf s}(\alpha)}\right)^{{\rm mult}(\alpha)}
	\\
\end{align*}

The Langlands duals of $D_4^{(3)}$ and $E_6^{(2)}$ are $G_2^{(1)}$ and $F_4^{(1)}$, respectively,
while $E_6^{(1)}$, $E_7^{(1)}$ and $E_8^{(1)}$ are self-dual. Then the specialization $\mathrm{F}_{\mathbf s}D(\Phi^\vee)$ is given by
\begin{align*}	
	&\mathrm F_{\mathbf s}\!\Bigl(\,\prod_{\alpha\in\Phi^{+}}
	\bigl(1-e^{-\alpha}\bigr)^{\dim\mathfrak g_{-\alpha}}\Bigr)\\
	&=\prod_{\alpha\in\Phi_+^0}\prod_{r\geq 0}\!\left(1-q^{\operatorname{ht}_{\mathbf s}(\alpha)+r \operatorname{ht}_{\mathbf s}(\delta)}\right)
	\prod_{\alpha\in\Phi_+^0}\prod_{r\geq 1}\!\left(1-q^{-\operatorname{ht}_{\mathbf s}(\alpha)+r \operatorname{ht}_{\mathbf s}(\delta)}\right)
	\prod_{r\geq 1}\!\left(1-q^{r \operatorname{ht}_{\mathbf s}(\delta)}\right)^{{\rm mult}(r\delta)}\\
	&=\prod_{\alpha\in\Phi_+^0}\prod_{r\geq 0}\!\left(1-q^{\operatorname{ht}_{\mathbf s}(\alpha)+r\operatorname{ht}_{\mathbf s}(\delta)}\right)
	\prod_{\alpha\in\Phi_+^0}\prod_{r\geq 0}\!\left(1-q^{\operatorname{ht}_{\mathbf s}(\delta)-\operatorname{ht}_{\mathbf s}(\alpha)+r\operatorname{ht}_{\mathbf s}(\delta)}\right)
	\prod_{r\geq 1}\!\left(1-q^{r\operatorname{ht}_{\mathbf s}(\delta)}\right)^{{\rm mult}(r\delta)}.
\end{align*}

Consequently, the principal specializations of $D(\Phi^\vee)$ in the exceptional affine types
considered in this paper are given in Table~\ref{tab:principal-specialization}.

\begin{longtable}{|l|l|P{0.33\textwidth}@{}|}
	\caption{Principal specialization of $D(\Phi^\vee)$}\label{tab:principal-specialization}\\
	\hline
	Type & $\Phi^\vee$ & $\mathrm{F}_1 D(\Phi^\vee)$ \\
	\hline
	\endfirsthead
	
	\hline
	Type & $\Phi^\vee$ & $\mathrm{F}_1 D(\Phi^\vee)$ \\
	\hline
	\endhead
	
	\hline
	\endfoot
	
	$G_2^{(1)}$ & $D_4^{(3)}$ & $(1;1)^{2}(1;6)(5;6)$ \\
	\hline
	$D_4^{(3)}$ & $G_2^{(1)}$ & $(1;1)^{2}(1;6)(5;6)$ \\
	\hline
	$F_4^{(1)}$ & $E_6^{(2)}$ & $(1;1)^{4}(5;6)(1;6)$ \\
	\hline
	$E_6^{(2)}$ & $F_4^{(1)}$ & $(1;1)^{4}(5;6)(1;6)$ \\
	\hline
	$E_6^{(1)}$ & $E_6^{(1)}$ & $(1;1)^{6}(1;6)(5;6)(4;12)(8;12)$ \\
	\hline
	$E_7^{(1)}$ & $E_7^{(1)}$ & $(1;1)^{7}(1;6)(5;6)(9;18)$ \\
	\hline
	$E_8^{(1)}$ & $E_8^{(1)}$ & $(1;1)^{8}(1;6)(5;6)(5;30)^{-1}(25;30)^{-1}$ \\
	\hline
\end{longtable}

Using the principal specializations in Table~\ref{tab:principal-specialization},
we obtain the following specializations of the normalized character
$e^{-\Lambda_0}\mathrm{ch}\,L(\Lambda_0)$.

\setlength{\tabcolsep}{4pt}
\renewcommand{\arraystretch}{1.12}

\begin{longtable}{|P{0.05\textwidth}|P{0.35\textwidth}|P{0.37\textwidth}@{}|}
	\caption{Principal specialization of $e^{-\Lambda_0}\mathrm{ch}\,L(\Lambda_0)$}\label{tab:specialization}\\
	\hline
	Type & $\mathrm{F}_{(\cdot)}\!\bigl(D(\Phi^\vee)\bigr)$
	& $\mathrm{F}_1\!\left(e^{-\Lambda_0}\mathrm{ch}\,L(\Lambda_0)\right)$ \\
	\hline
	\endfirsthead
	\hline
	Type & $\mathrm{F}_{(\cdot)}\!\bigl(D(\Phi^\vee)\bigr)$
	& $\mathrm{F}_1\!\left(e^{-\Lambda}\mathrm{ch}\,L(\Lambda_0)\right)$ \\
	\hline
	\endhead
	\hline
	\endfoot
	
	$G_2^{(1)}$ &
	$\mathrm{F}_{(2,1,1)}=(1;1)^2(6;15)^{-1}(9;15)^{-1}$ &
	$(1;6)^{-1}(5;6)^{-1}(6;15)^{-1}(9;15)^{-1}$ \\
	\hline
	
	$D_4^{(3)}$ &
	$\mathrm{F}_{(2,1,1)}=(1;1)^{2}$ &
	$(5;6)^{-1}(1;6)^{-1}$ \\
	\hline
	
	$F_4^{(1)}$ &
	$\mathrm{F}_{(2,1,1,1,1)}=(1;1)^4(8;20)^{-1}(12;20)^{-1}$ &
	$(8;20)^{-1}(12;20)^{-1}(1;6)^{-1}(5;6)^{-1}$ \\
	\hline
	
	$E_6^{(2)}$ &
	$\mathrm{F}_{(2,1,1,1,1)}=(1;1)^{4}$ &
	$(5;6)^{-1}(1;6)^{-1}$ \\
	\hline
	
	$E_6^{(1)}$ &
	$\mathrm{F}_{(2,1,1,1,1,1,1)}=(1;1)^6$ &
	$(1;6)^{-1}(5;6)^{-1}(4;12)^{-1}(8;12)^{-1}$ \\
	\hline
	
	$E_7^{(1)}$ &
	$\mathrm{F}_{(2,1,1,1,1,1,1,1)}=(1;1)^7$ &
	$(1;6)^{-1}(5;6)^{-1}(9;18)^{-1}$ \\
	\hline
	
	$E_8^{(1)}$ &
	$\mathrm{F}_{(2,1,1,1,1,1,1,1,1)}=(1;1)^8$ &
	$(1;30)^{-1}(7;30)^{-1}(11;30)^{-1}(13;30)^{-1}$
	$(17;30)^{-1}(19;30)^{-1}(23;30)^{-1}(29;30)^{-1}$ \\
	\hline
	\endlastfoot
\end{longtable}

\section{Congruence, Initial and Difference conditions}\label{sec:cidcon}

In this section we take the principal specialization of the sum side
in~\eqref{eq:coloredKMN}. For all exceptional affine types considered here we have
$d_0=1$. Recall that in~\eqref{eq:coloredKMN} we set $q=e^{-\delta}$.

Let $\lambda=(\lambda_1^{b_1},\ldots,\lambda_s^{b_s})\in\mathcal{P}_\phi^{\gg}$.
Then
\[
\mathrm{F}_1\!\bigl(C(\lambda)\,e^{-\delta|\lambda|}\bigr)
=
q^{-\sum_{k=1}^{s}\operatorname{ht}\bigl(\overline{\mathrm{wt}}(b_k)\bigr)}\,
q^{\operatorname{ht}(\delta)\sum_{k=1}^{s}\lambda_k}
=
q^{\sum_{k=1}^{s}\pi_k},
\]
where we define
\begin{equation}\label{eq:pik}
	\pi_k
	:=\operatorname{ht}(\delta)\lambda_k-\operatorname{ht}\bigl(\overline{\mathrm{wt}}(b_k)\bigr)
	\qquad (1\le k\le s).
\end{equation}

Recall the set of
$B$-colored integers $\mathcal{Z}_B:=\{\,k^b \mid k\in\mathbb Z,\ b\in B\,\}$. 
We further set $\mathcal{Z}_B^{(\mathrm{fin})}:=\bigsqcup_{s\ge 1}\mathcal{Z}_B^{\,s}$,
the set of all finite (nonempty) sequences of $B$-colored integers.

Define the map
\begin{equation}\label{eq:defPhi}
	\Phi:\mathcal{P}_\phi^{\gg}\longrightarrow \mathcal{Z}_B^{(\mathrm{fin})},\qquad
	\Phi(\lambda)=(\pi_1^{b_1},\ldots,\pi_s^{b_s}),
\end{equation}
and set $\mathcal{P}'_\phi:=\Phi(\mathcal{P}_\phi^{\gg})$.

For $b\in B$, let $\gamma(b)$ be the unique representative of
$-\operatorname{ht}\bigl(\overline{\mathrm{wt}}(b)\bigr)\bmod \operatorname{ht}(\delta)$
in $\{0,1,\ldots,\operatorname{ht}(\delta)-1\}$. Equivalently,
\begin{equation}\label{eq:gammab}
	\gamma(b)=
	\begin{cases}
		-\operatorname{ht}\bigl(\overline{\mathrm{wt}}(b)\bigr),
		& \operatorname{ht}\bigl(\overline{\mathrm{wt}}(b)\bigr)\le 0,\\[2pt]
		\operatorname{ht}(\delta)-\operatorname{ht}\bigl(\overline{\mathrm{wt}}(b)\bigr),
		& \operatorname{ht}\bigl(\overline{\mathrm{wt}}(b)\bigr)>0.
	\end{cases}
\end{equation}
Then \eqref{eq:pik} implies that every part $\pi_k$ of color $b_k$ satisfies the
congruence $\pi_k\equiv \gamma(b_k)\pmod{\operatorname{ht}(\delta)}$.

The explicit congruence data for each type are listed in
Appendix~\ref{sec:congruence-condition}.

\begin{lemma}\label{lem:injPhi}
	The map $\Phi$ in \eqref{eq:defPhi} is injective.
\end{lemma}

\begin{proof}
	Let $\pi=(\pi_k^{b_k})_{1\le k\le s}\in\mathcal P'_\phi$ be a colored partition.
	We claim that $\pi$ determines a unique
	$\lambda=(\lambda_k^{b_k})_{1\le k\le s}\in\mathcal{P}_\phi^{\gg}$, namely via
	\[
	\lambda_k:=\frac{\pi_k+\operatorname{ht}\bigl(\overline{\mathrm{wt}}(b_k)\bigr)}
	{\operatorname{ht}(\delta)}
	\qquad(1\le k\le s).
	\]
	By the congruence condition $\pi_k\equiv \gamma(b_k)\pmod{\operatorname{ht}(\delta)}$
	and the definition of $\gamma$ in \eqref{eq:gammab}, we have
	\[
	\pi_k+\operatorname{ht}\bigl(\overline{\mathrm{wt}}(b_k)\bigr)\equiv 0
	\pmod{\operatorname{ht}(\delta)},
	\]
	so each $\lambda_k$ is an integer. It then follows from \eqref{eq:pik} that
	$\Phi(\lambda)=\pi$. Therefore $\Phi$ is injective.
\end{proof}

By the energy admissibility~\eqref{eq:energy-condition} and the definition~\eqref{eq:pik}, for
$1\le k\le s-1$ we obtain
\begin{equation}\label{eq:diffpi}
	\pi_k-\pi_{k+1}\ge
	\operatorname{ht}(\delta)\,H(b_{k+1}\otimes b_k)
	+\operatorname{ht}(\overline{\mathrm{wt}}(b_{k+1}))
	-\operatorname{ht}(\overline{\mathrm{wt}}(b_k))
	=
	F(b_{k+1}\otimes b_k).
\end{equation}
By Lemma~\ref{lem:Fgeq0}, the right-hand side is nonnegative.  Hence for any
$\pi=(\pi_1^{b_1},\dots,\pi_s^{b_s})\in \mathcal P'_\phi$ we have $\pi_1\ge \pi_2\ge \cdots\ge \pi_s\ge 0$.
In particular, all zero parts  occur at the end; thus $\sum_{k=1}^{s}\pi_k=\sum_{k=1}^{\ell(\pi)}\pi_k$.

By Lemma~\ref{lem:injPhi}, we can re-index the principal specialization of the left-hand side of
\eqref{eq:coloredKMN} by $\pi=\Phi(\lambda)$ and obtain
\begin{equation}\label{eq:F1left}
	\mathrm{F}_1\!\left(\sum_{\lambda\in\mathcal{P}_\phi^{\gg}} C(\lambda)e^{-\delta|\lambda|}\right)
	=
	\sum_{\pi\in\mathcal{P}'_\phi} q^{\sum_{k=1}^{\ell(\pi)}\pi_k}.
\end{equation}

Moreover, since $\lambda$ is grounded we have
$\pi_s=0$ of color $\phi$, so every $\pi\in\mathcal{P}'_g$ is a generalized
$B$-colored partition with nonnegative parts.

Define a relation $\gg_F$ on $\mathcal{Z}_B$ by
\[
\pi^{\,b}\gg_F (\pi')^{\,b'}
\quad\Longleftrightarrow\quad
\pi-\pi'\ge F(b'\otimes b).
\]
Then~\eqref{eq:diffpi} is equivalent to
\[
\pi_k^{\,b_k}\gg_F \pi_{k+1}^{\,b_{k+1}}
\qquad(1\le k\le s-1),
\]
so every $\pi\in\mathcal{P}'_\phi$ satisfies the $F$-difference conditions.

Finally, we encode the difference conditions in the \emph{difference matrix}
\[
M=(M_{b',b})_{b',b\in B},
\qquad
M_{b',b}:=F(b'\otimes b).
\]
Then the admissibility condition reads
\[
\pi_k-\pi_{k+1}\ge M_{b_{k+1},b_k}\qquad(1\le k\le s-1).
\]

\begin{lemma}\label{lem:Ftildef}
	Let $i\in I$ and $b_1,b_2\in B$. If $\widetilde f_i(b_1\otimes b_2)\ne 0$, then
	\[
	F\!\left(\widetilde f_i(b_1\otimes b_2)\right)
	=
	\begin{cases}
		F(b_1\otimes b_2)-1, & \text{if }\ \varphi_i(b_1)>\varepsilon_i(b_2),\\[4pt]
		F(b_1\otimes b_2)+1, & \text{if }\ \varphi_i(b_1)\le \varepsilon_i(b_2).
	\end{cases}
	\]
\end{lemma}

\begin{proof}
	By the tensor product rule, if $\varphi_i(b_1)>\varepsilon_i(b_2)$ then
	$\widetilde f_i(b_1\otimes b_2)=(\widetilde f_i b_1)\otimes b_2$, whereas if
	$\varphi_i(b_1)\le\varepsilon_i(b_2)$ then
	$\widetilde f_i(b_1\otimes b_2)=b_1\otimes(\widetilde f_i b_2)$.
	
	\smallskip
	\noindent\emph{Case $i\ne 0$.}
	By Definition~\ref{def:energy}, the energy $H$ is unchanged under $\widetilde f_i$ for $i\ne 0$.
	Moreover, $\overline{\mathrm{wt}}(\widetilde f_i b)=\overline{\mathrm{wt}}(b)-\alpha_i$, hence $\operatorname{ht}\bigl(\overline{\mathrm{wt}}(\widetilde f_i b)\bigr)
	=
	\operatorname{ht}\bigl(\overline{\mathrm{wt}}(b)\bigr)-1$.
	
	Therefore, we obtain $F\bigl((\widetilde f_i b_1)\otimes b_2\bigr)=F(b_1\otimes b_2)-1$ and $F\bigl(b_1\otimes(\widetilde f_i b_2)\bigr)=F(b_1\otimes b_2)+1$,
	which gives the two subcases for $i\ne 0$.
	
	\smallskip
	\noindent\emph{Case $i=0$.}
	By Definition~\ref{def:energy}, $H$ changes by $-1$ in the first subcase and by $+1$
	in the second subcase. By the crystal structure in \eqref{eq:crys}, the classical weight
	changes by $\overline{\mathrm{wt}}(\widetilde f_0 b)=\overline{\mathrm{wt}}(b)+\theta$, hence $\operatorname{ht}\bigl(\overline{\mathrm{wt}}(\widetilde f_0 b)\bigr)
	=
	\operatorname{ht}\bigl(\overline{\mathrm{wt}}(b)\bigr)+\operatorname{ht}(\theta)$.
	Since $\operatorname{ht}(\delta)-\operatorname{ht}(\theta)=1$, it follows that $F\bigl((\widetilde f_0 b_1)\otimes b_2\bigr)=F(b_1\otimes b_2)-1$ and $F\bigl(b_1\otimes(\widetilde f_0 b_2)\bigr)=F(b_1\otimes b_2)+1$.
	This proves the lemma.
\end{proof}

Since $B\otimes B$ is connected, the update rule in Lemma~\ref{lem:Ftildef}, together with the
normalization $F(\phi\otimes\phi)=0$, uniquely determines $F(b\otimes b')$ for all $b,b'\in B$.

Viewing the underlying crystal graph of $B\otimes B$ with $i$-arrows given by the Kashiwara
operators $\widetilde f_i$, one can compute $F$ by a breadth-first traversal starting from
$\phi\otimes\phi$ and propagating values along each arrow using Lemma~\ref{lem:Ftildef}.

Importantly, for each affine type $X$ under consideration, this procedure requires only the
crystal data (the operators $\widetilde f_i$ together with $\varepsilon_i$ and $\varphi_i$),
and therefore produces the full difference matrix $M_X=\bigl(F(b_1\otimes b_2)\bigr)_{b_1,b_2\in B}$
without explicitly evaluating the energy function $H$.

Starting from the normalization $F(\phi\otimes\phi)=0$, suppose that $F(b_1\otimes b_2)$ has been
determined for some $b_1,b_2\in B$. For each $i\in I$, set $D_i(b_1\otimes b_2):=\varphi_i(b_1)-\varepsilon_i(b_2)$.
By the tensor product rule and Lemma~\ref{lem:Ftildef}, $F$ can be propagated along the
$\widetilde f_i$-arrows in the connected crystal graph of $B\otimes B$ via the following local
updates:
\begin{itemize}
	\item if $D_i(b_1\otimes b_2)>0$ and $\widetilde f_i b_1\ne 0$, define $F\bigl((\widetilde f_i b_1)\otimes b_2\bigr)=F(b_1\otimes b_2)-1$;
	\item if $D_i(b_1\otimes b_2)\le 0$ and $\widetilde f_i b_2\ne 0$, define $F\bigl(b_1\otimes(\widetilde f_i b_2)\bigr)=F(b_1\otimes b_2)+1$.
\end{itemize}
Iterating these updates while traversing the connected graph of $B\otimes B$ determines $F$ on
every vertex, and hence yields the full modified difference matrix $M_X$.

\medskip

In practice, computing $M_X$ requires only the root-theoretic input specifying $B$
(the index set $I$, the set of (short) positive roots, the highest (short) root, and the
simple roots). From this data one constructs $B$, traverses the connected graph of
$B\otimes B$, propagates $F$ using the local update rules above, and thereby obtains $M_X$;
see Algorithm~\ref{alg:computeF}.

\begin{algorithm}[H]
	\caption{Computing the difference matrix $M=(F(b_1\otimes b_2))$ from crystal data}
	\label{alg:computeF}
	\begin{algorithmic}[1]
		\Require A level-one perfect crystal $B=B(\theta)\sqcup B(0)$ with ground element $\phi\in B(0)$;
		the Kashiwara operators $\widetilde f_i$ on $B$; and the functions $\varepsilon_i,\varphi_i$
		for $i\in I$.
		\Ensure The matrix $M$ with entries $M_{b_1,b_2}=F(b_1\otimes b_2)$.
		
		\State Construct the vertex set of $B$ in the prescribed order (ground element, positive roots,
		simple-root elements, negative roots).
		\State Construct the directed graph of $B\otimes B$ using the induced arrows $\widetilde f_i$.
		\State Precompute $\varepsilon_i$ and $\varphi_i$ on $B$ and lift them to $B\otimes B$.
		
		\State Initialize $F(\phi\otimes\phi)\gets 0$ and a queue $Q\gets\{\phi\otimes\phi\}$.
		\While{$Q$ is not empty}
		\State Pop $b_1\otimes b_2$ from $Q$.
		\For{each $i\in I$}
		\If{$\varphi_i(b_1)>\varepsilon_i(b_2)$ and $\widetilde f_i b_1\ne 0$}
		\State Set $b_1'\gets \widetilde f_i b_1$ and
		$F(b_1'\otimes b_2)\gets F(b_1\otimes b_2)-1$.
		\State Push $b_1'\otimes b_2$ into $Q$ if it is newly assigned.
		\ElsIf{$\widetilde f_i b_2\ne 0$}
		\State Set $b_2'\gets \widetilde f_i b_2$ and
		$F(b_1\otimes b_2')\gets F(b_1\otimes b_2)+1$.
		\State Push $b_1\otimes b_2'$ into $Q$ if it is newly assigned.
		\EndIf
		\EndFor
		\EndWhile
		\State \Return $M=(F(b_1\otimes b_2))_{b_1,b_2\in B}$.
	\end{algorithmic}
\end{algorithm}

For completeness, the matrices $M_X$ are tabulated explicitly, type by type, in
Appendix~\ref{sec:energy-matrix}.

\vskip 2mm

\begin{lemma}\label{lem:pispis-1}
	For any $\pi=(\pi_k^{b_k})_{1\le k\le s}\in\mathcal{P}'_\phi$ with $s\ge 2$, we have
	$\pi_{s-1}^{b_{s-1}}\ne 0^{\phi}$.
\end{lemma}

\begin{proof}
	Assume $s\ge 2$ and suppose, for contradiction, that $\pi_{s-1}^{b_{s-1}}=0^\phi$.
	Then $b_{s-1}=\phi$ and $\pi_{s-1}=0$. Substituting into \eqref{eq:pik} gives
	\[
	0=\pi_{s-1}
	=
	\operatorname{ht}(\delta)\lambda_{s-1}-\operatorname{ht}\bigl(\overline{\mathrm{wt}}(\phi)\bigr)
	=
	\operatorname{ht}(\delta)\lambda_{s-1},
	\]
	so $\lambda_{s-1}=0$. Therefore $\lambda_{s-1}^{b_{s-1}}=0^\phi$, contradicting the
	grounded condition in the definition of $\mathcal{P}_\phi^\gg$.
\end{proof}

\begin{lemma}\label{lem:pi-penultimate-positive}
	Let $\lambda=(\lambda_1^{b_1},\dots,\lambda_s^{b_s})\in\mathcal{P}_\phi^\gg$, and let
	$\pi=\Phi(\lambda)=(\pi_1^{b_1},\dots,\pi_s^{b_s})$ be defined by \eqref{eq:pik}.
	Then $\pi_{s-1}>0$.
\end{lemma}

\begin{proof}
	By Lemma~\ref{lem:lambda-penultimate-positive} we have $\lambda_{s-1}\ge 1$. Hence
	\[
	\pi_{s-1}
	=\operatorname{ht}(\delta)\lambda_{s-1}-\operatorname{ht}\bigl(\overline{\mathrm{wt}}(b_{s-1})\bigr)
	\ge \operatorname{ht}(\delta)-\operatorname{ht}(\theta)=1,
	\]
	where we used \eqref{eq:hbound} to bound
	$\operatorname{ht}\bigl(\overline{\mathrm{wt}}(b_{s-1})\bigr)\le \operatorname{ht}(\theta)$.
\end{proof}

Therefore, by Lemmas~\ref{lem:pispis-1}--\ref{lem:pi-penultimate-positive}, the zero part of every element
$\pi=(\pi_1^{b_1},\dots,\pi_s^{b_s})\in\mathcal{P}'_\phi$ is unique.

\begin{lemma}\label{lem:forbiddenpart}
	Let $\pi=(\pi_k^{b_k})_{1\le k\le s}\in\mathcal{P}'_\phi$. No part of $\pi$ can satisfy
	$b_k\ne \phi$ and $\pi_k=-\operatorname{ht}\bigl(\overline{\mathrm{wt}}(b_k)\bigr)$.
\end{lemma}

\begin{proof}
	If $\pi_k=-\operatorname{ht}\bigl(\overline{\mathrm{wt}}(b_k)\bigr)$, then by \eqref{eq:pik} we have
	$\lambda_k=0$. Since $b_k\ne\phi$, this gives a zero part of $\lambda$ with non-ground color,
	contradicting the groundedness condition defining $\mathcal{P}_\phi^\gg$.
\end{proof}

By Lemmas~\ref{lem:pispis-1}--\ref{lem:forbiddenpart}, the set of forbidden parts in
$\mathcal{P}'_\phi$ is
\begin{equation*}
	\mathrm{For}_\phi=
	\{\,0^b\mid b\in B\setminus\{\phi\}\,\}
	\ \cup\
	\left\{\,
	\bigl(-\operatorname{ht}(\overline{\mathrm{wt}}(b))\bigr)^{\,b}
	\ \middle|\ b\in B\setminus\{\phi\}
	\,\right\}.
\end{equation*}

Since every $\pi\in\mathcal{P}'_\phi$ ends with $0^\phi$ and has $\pi_{s-1}>0$, it suffices to
exclude the parts in $\mathrm{For}_\phi$.

\vskip 2mm

From the tables in Appendix \ref{sec:congruence-condition}, we obtain the following forbidden initial colored parts ${\rm ICon}_X$ of each type:

\begin{longtable}{|c|p{0.8\textwidth}|}
	\caption{The initial conditions}\label{tab:initial-condition}\\	
	\hline
	Type $X$ & Forbidden initial colored parts ${\rm ICon}_X$ \\
	\hline
	\endfirsthead
	
	\hline
	Type $X$ & Forbidden initial colored parts ${\rm ICon}_X$ \\
	\hline
	\endhead
	
	$D_4^{(3)}$ &
	\begin{tabular}[t]{@{}l@{}}
		$1^{\overline{(10)}},\;
		2^{\overline{(11)}},\;
		3^{\overline{(21)}}$
	\end{tabular}
	\\ \hline
	
	$G_2^{(1)}$ &
	\begin{tabular}[t]{@{}l@{}}
		$1^{\overline{(10)}},\;
		1^{\overline{(01)}},\;
		2^{\overline{(11)}},\;
		3^{\overline{(12)}},\;
		4^{\overline{(13)}},\;
		5^{\overline{(23)}}$
	\end{tabular}
	\\ \hline
	
	$E_6^{(2)}$ &
	\begin{tabular}[t]{@{}l@{}}
		$1^{\overline{(1000)}},\ 
		1^{\overline{(0100)}},\ 
		2^{\overline{(0110)}},\ 
		2^{\overline{(1100)}},\ 
		3^{\overline{(0111)}},\ 
		3^{\overline{(1110)}},\
		4^{\overline{(1210)}},\ 
		4^{\overline{(1111)}},\ 
		5^{\overline{(1211)}},\ 
		6^{\overline{(1221)}},$\\ 
		$7^{\overline{(1321)}},\ 
		8^{\overline{(2321)}}$
	\end{tabular}
	\\ \hline
	
	$F_4^{(1)}$ &
	\begin{tabular}[t]{@{}l@{}}
		$1^{\overline{(1000)}},\ 1^{\overline{(0100)}},\ 1^{\overline{(0010)}},\ 1^{\overline{(0001)}},\
		2^{\overline{(1100)}},\ 2^{\overline{(0110)}},\ 2^{\overline{(0011)}},\
		3^{\overline{(1110)}},\ 3^{\overline{(0120)}},\ 3^{\overline{(0111)}},$\\
		$4^{\overline{(1120)}},\ 4^{\overline{(1111)}},\ 4^{\overline{(0121)}},\
		5^{\overline{(1220)}},\ 5^{\overline{(1121)}},\ 5^{\overline{(0122)}},\
		6^{\overline{(1221)}},\ 6^{\overline{(1122)}},\
		7^{\overline{(1231)}},\ 7^{\overline{(1222)}},$\\
		$8^{\overline{(1232)}},\ 
		9^{\overline{(1242)}},\ 
		10^{\overline{(1342)}},\ 
		11^{\overline{(2342)}}$
	\end{tabular}
	\\ \hline
	
	$E_6^{(1)}$ &
	\begin{tabular}[t]{@{}l@{}}
		$1^{\overline{(100000)}},\ 1^{\overline{(010000)}},\ 1^{\overline{(001000)}},\ 
		1^{\overline{(000100)}},\ 1^{\overline{(000010)}},\ 1^{\overline{(000001)}},\
		2^{\overline{(101000)}},\ 2^{\overline{(010100)}},$\\ $2^{\overline{(001100)}},\ 
		2^{\overline{(000110)}},\ 2^{\overline{(000011)}},\
		3^{\overline{(101100)}},\ 3^{\overline{(011100)}},\ 3^{\overline{(010110)}},\ 
		3^{\overline{(001110)}},\ 3^{\overline{(000111)}},$\\
		$4^{\overline{(111100)}},\ 4^{\overline{(101110)}},\ 4^{\overline{(011110)}},\ 
		4^{\overline{(010111)}},\ 4^{\overline{(001111)}},\
		5^{\overline{(111110)}},\ 5^{\overline{(101111)}},\ 5^{\overline{(011210)}},$\\ 
		$5^{\overline{(011111)}},\
		6^{\overline{(111210)}},\ 6^{\overline{(111111)}},\ 6^{\overline{(011211)}},\
		7^{\overline{(112210)}},\ 7^{\overline{(111211)}},\ 7^{\overline{(011221)}},\
		8^{\overline{(112211)}},$\\ $8^{\overline{(111221)}},\
		9^{\overline{(112221)}},\ 
		10^{\overline{(112321)}},\ 
		11^{\overline{(122321)}}$
	\end{tabular}
	\\ \hline
	
	$E_7^{(1)}$ &
	\begin{tabular}[t]{@{}l@{}}
		$1^{\overline{(1000000)}},\ 1^{\overline{(0100000)}},\ 1^{\overline{(0010000)}},\ 
		1^{\overline{(0001000)}},\ 1^{\overline{(0000100)}},\ 1^{\overline{(0000010)}},\ 1^{\overline{(0000001)}},\
		2^{\overline{(1010000)}},$\\ $2^{\overline{(0101000)}},\ 2^{\overline{(0011000)}},\ 
		2^{\overline{(0001100)}},\ 2^{\overline{(0000110)}},\ 2^{\overline{(0000011)}},\
		3^{\overline{(1011000)}},\ 3^{\overline{(0111000)}},\ 3^{\overline{(0101100)}},$\\ 
		$3^{\overline{(0011100)}},\ 3^{\overline{(0001110)}},\ 3^{\overline{(0000111)}},\
		4^{\overline{(1111000)}},\ 4^{\overline{(1011100)}},\ 4^{\overline{(0111100)}},\ 
		4^{\overline{(0101110)}},\ 4^{\overline{(0011110)}},$\\ $4^{\overline{(0001111)}},\
		5^{\overline{(1111100)}},\ 5^{\overline{(1011110)}},\ 5^{\overline{(0112100)}},\ 
		5^{\overline{(0111110)}},\ 5^{\overline{(0101111)}},\ 5^{\overline{(0011111)}},\
		6^{\overline{(1112100)}},$\\ $6^{\overline{(1111110)}},\ 6^{\overline{(1011111)}},\ 
		6^{\overline{(0112110)}},\ 6^{\overline{(0111111)}},\
		7^{\overline{(1122100)}},\ 7^{\overline{(1112110)}},\ 7^{\overline{(1111111)}},\ 
		7^{\overline{(0112210)}},$\\ $7^{\overline{(0112111)}},\
		8^{\overline{(1122110)}},\ 8^{\overline{(1112210)}},\ 8^{\overline{(1112111)}},\ 
		8^{\overline{(0112211)}},\
		9^{\overline{(1122210)}},\ 9^{\overline{(1122111)}},\ 9^{\overline{(1112211)}},$\\ 
		$9^{\overline{(0112221)}},\
		10^{\overline{(1123210)}},\ 10^{\overline{(1122211)}},\ 10^{\overline{(1112221)}},\
		11^{\overline{(1223210)}},\ 11^{\overline{(1123211)}},\ 11^{\overline{(1122221)}},$\\
		$12^{\overline{(1223211)}},\ 12^{\overline{(1123221)}},\
		13^{\overline{(1223221)}},\ 13^{\overline{(1123321)}},\
		14^{\overline{(1223321)}},\ 
		15^{\overline{(1224321)}},\ 
		16^{\overline{(1234321)}},$\\ 
		$17^{\overline{(2234321)}}$
	\end{tabular}
	\\ \hline
	
	$E_8^{(1)}$ &
	\begin{tabular}[t]{@{}l@{}}
		$1^{\overline{(00000001)}},\ 1^{\overline{(00000010)}},\ 1^{\overline{(00000100)}},\
		1^{\overline{(00001000)}},\ 1^{\overline{(00010000)}},\ 1^{\overline{(00100000)}},\
		1^{\overline{(01000000)}},$\\ $1^{\overline{(10000000)}},\
		2^{\overline{(00000011)}},\ 2^{\overline{(00000110)}},\ 2^{\overline{(00001100)}},\
		2^{\overline{(00011000)}},\ 2^{\overline{(00110000)}},\ 2^{\overline{(01010000)}},$\\
		$2^{\overline{(10100000)}},\
		3^{\overline{(00000111)}},\ 3^{\overline{(00001110)}},\ 3^{\overline{(00011100)}},\
		3^{\overline{(00111000)}},\ 3^{\overline{(01011000)}},\ 3^{\overline{(01110000)}},$\\
		$3^{\overline{(10110000)}},\
		4^{\overline{(00001111)}},\ 4^{\overline{(00011110)}},\ 4^{\overline{(00111100)}},\
		4^{\overline{(01011100)}},\ 4^{\overline{(01111000)}},\ 4^{\overline{(10111000)}},$\\
		$4^{\overline{(11110000)}},\
		5^{\overline{(00011111)}},\ 5^{\overline{(00111110)}},\ 5^{\overline{(01011110)}},\
		5^{\overline{(01111100)}},\ 5^{\overline{(01121000)}},\ 5^{\overline{(10111100)}},$\\
		$5^{\overline{(11111000)}},\
		6^{\overline{(00111111)}},\ 6^{\overline{(01011111)}},\ 6^{\overline{(01111110)}},\
		6^{\overline{(01121100)}},\ 6^{\overline{(10111110)}},\ 6^{\overline{(11111100)}},$\\
		$6^{\overline{(11121000)}},\
		7^{\overline{(01111111)}},\ 7^{\overline{(01121110)}},\ 7^{\overline{(01122100)}},\
		7^{\overline{(10111111)}},\ 7^{\overline{(11111110)}},\ 7^{\overline{(11121100)}},$\\
		$7^{\overline{(11221000)}},\
		8^{\overline{(01121111)}},\ 8^{\overline{(01122110)}},\ 8^{\overline{(11111111)}},\
		8^{\overline{(11121110)}},\ 8^{\overline{(11122100)}},\ 8^{\overline{(11221100)}},$\\
		$9^{\overline{(01122111)}},\ 9^{\overline{(01122210)}},\ 9^{\overline{(11121111)}},\
		9^{\overline{(11122110)}},\ 9^{\overline{(11221110)}},\ 9^{\overline{(11222100)}},\
		10^{\overline{(01122211)}},$\\ $10^{\overline{(11122111)}},\ 10^{\overline{(11122210)}},\
		10^{\overline{(11221111)}},\ 10^{\overline{(11222110)}},\ 10^{\overline{(11232100)}},\
		11^{\overline{(01122221)}},$\\ $11^{\overline{(11122211)}},\ 11^{\overline{(11222111)}},\
		11^{\overline{(11222210)}},\ 11^{\overline{(11232110)}},\ 11^{\overline{(12232100)}},\
		12^{\overline{(11122221)}},$\\ $12^{\overline{(11222211)}},\ 12^{\overline{(11232111)}},\
		12^{\overline{(11232210)}},\ 12^{\overline{(12232110)}},\
		13^{\overline{(11222221)}},\ 13^{\overline{(11232211)}},$\\ $13^{\overline{(11233210)}},\
		13^{\overline{(12232111)}},\ 13^{\overline{(12232210)}},\
		14^{\overline{(11232221)}},\ 14^{\overline{(11233211)}},\ 14^{\overline{(12232211)}},$\\
		$14^{\overline{(12233210)}},\
		15^{\overline{(11233221)}},\ 15^{\overline{(12232221)}},\ 15^{\overline{(12233211)}},\
		15^{\overline{(12243210)}},\
		16^{\overline{(11233321)}},$\\ $16^{\overline{(12233221)}},\ 16^{\overline{(12243211)}},\
		16^{\overline{(12343210)}},\
		17^{\overline{(12233321)}},\ 17^{\overline{(12243221)}},\ 17^{\overline{(12343211)}},$\\
		$17^{\overline{(22343210)}},\
		18^{\overline{(12243321)}},\ 18^{\overline{(12343221)}},\ 18^{\overline{(22343211)}},\
		19^{\overline{(12244321)}},\ 19^{\overline{(12343321)}},$\\ $19^{\overline{(22343221)}},\
		20^{\overline{(12344321)}},\ 20^{\overline{(22343321)}},\
		21^{\overline{(12354321)}},\ 21^{\overline{(22344321)}},\
		22^{\overline{(13354321)}},$\\ $22^{\overline{(22354321)}},\
		23^{\overline{(22454321)}},\ 23^{\overline{(23354321)}},\
		24^{\overline{(23454321)}},\
		25^{\overline{(23464321)}},\
		26^{\overline{(23465321)}},$\\
		$27^{\overline{(23465421)}},\
		28^{\overline{(23465431)}},\
		29^{\overline{(23465432)}}$
	\end{tabular}
	\\ \hline
	
\end{longtable}

\section{The partition identity}\label{sec:pi}

From Table~\ref{tab:specialization}, together with the extra factor
$(q^{\operatorname{ht}(\delta)};q^{\operatorname{ht}(\delta)})_\infty^{-1}$
arising from $(q;q)_\infty^{-1}$ in \eqref{eq:coloredKMN} under principal specialization,
we define, for each affine type $X$, a set $P_X$ by
\[
P_X
:= \bigsqcup_{t=1}^{T_X}\{\,m\in\mathbb Z_{>0}\mid m\equiv r_t \pmod{N_t}\,\},
\]
with congruence data $(r_t,N_t)$ as listed in Table~\ref{tab:PX-congruence-data}.

\begin{table}[H]
	\centering
	\setlength{\tabcolsep}{6pt}
	\renewcommand{\arraystretch}{1.15}
	\caption{Congruence data $(r_t,N_t)$ for the sets $P_X$}
	\label{tab:PX-congruence-data}
	\begin{tabular}{|c|c|l|}
		\hline
		Type $X$ & $T_X$ & Congruence data $\{(r_t,N_t)\}_{t=1}^{T_X}$ \\
		\hline
		$G_2^{(1)}$ & $5$ &
		$\{(0,6),(1,6),(5,6),(6,15),(9,15)\}$ \\
		\hline
		$D_4^{(3)}$ & $3$ &
		$\{(1,6),(5,6),(0,4)\}$ \\
		\hline
		$E_6^{(2)}$ & $3$ &
		$\{(1,6),(5,6),(0,9)\}$ \\
		\hline
		$F_4^{(1)}$ & $5$ &
		$\{(1,6),(5,6),(8,20),(12,20),(0,12)\}$ \\
		\hline
		$E_6^{(1)}$ & $5$ &
		$\{(1,6),(5,6),(0,12),(4,12),(8,12)\}$ \\
		\hline
		$E_7^{(1)}$ & $4$ &
		$\{(1,6),(5,6),(0,18),(9,18)\}$ \\
		\hline
		$E_8^{(1)}$ & $9$ &
		$\{(0,30),(1,30),(7,30),(11,30),(13,30),(17,30),(19,30),(23,30),(29,30)\}$ \\
		\hline
	\end{tabular}
\end{table}

For each $m\in\mathbb Z_{>0}$, define $a_m^{(X)} \;:=\; \sum_{t=1}^{T_X}\mathbf{1}_{\,m\equiv r_t\;(\mathrm{mod}\,N_t)}$. Equivalently, $a_m^{(X)}$ is the multiplicity of $m$ when $P_X$ is viewed as a multiset,
i.e.\ each congruence class in Table~\ref{tab:PX-congruence-data} contributes one copy of
every integer it contains. 

In particular, the sequence $(a_m^{(X)})_{m\ge 1}$ is uniquely
determined by the Euler-type product identity
\[
\prod_{t=1}^{T_X}\ \prod_{\substack{k\ge 1\\ k\equiv r_t\;(\mathrm{mod}\,N_t)}} \frac{1}{1-q^k}
\;=\;
\prod_{m\ge 1}(1-q^m)^{-a_m^{(X)}}.
\]
The product on the right is the generating function for partitions in which a part of size
$m$ comes in $a_m^{(X)}$ distinguishable species (equivalently, colors).

If $X\in\{D_4^{(3)},E_6^{(2)},E_6^{(1)},E_7^{(1)},E_8^{(1)}\}$, then the residue classes
$\{\,m\equiv r_t\pmod{N_t}\,\}_{t=1}^{T_X}$ are pairwise disjoint. Hence $a_m^{(X)}\in\{0,1\}$ for all $m\in\mathbb Z_{>0}$,
and moreover $a_m^{(X)}=1$ if and only if $m\in P_X$.

\smallskip

If $X\in\{G_2^{(1)},F_4^{(1)}\}$, then overlaps occur and one can have $a_m^{(X)}>1$.
More precisely:
\begin{itemize}
	\item If $X=G_2^{(1)}$, then
	\[
	a_m^{(G_2^{(1)})}
	= \mathbf 1_{\,m\equiv 0\;(\mathrm{mod}\,6)}
	+ \mathbf 1_{\,m\equiv 1\;(\mathrm{mod}\,6)}
	+ \mathbf 1_{\,m\equiv 5\;(\mathrm{mod}\,6)}
	+ \mathbf 1_{\,m\equiv 6\;(\mathrm{mod}\,15)}
	+ \mathbf 1_{\,m\equiv 9\;(\mathrm{mod}\,15)}
	\in\{0,1,2\}.
	\]
	In particular, $a_m^{(G_2^{(1)})}=2$ if and only if
	\[
	m\equiv 6 \pmod{30}\quad\text{or}\quad m\equiv 24 \pmod{30}.
	\]
	
	\item If $X=F_4^{(1)}$, then
	\[
	a_m^{(F_4^{(1)})}
	= \mathbf 1_{\,m\equiv 1\;(\mathrm{mod}\,6)}
	+ \mathbf 1_{\,m\equiv 5\;(\mathrm{mod}\,6)}
	+ \mathbf 1_{\,m\equiv 8\;(\mathrm{mod}\,20)}
	+ \mathbf 1_{\,m\equiv 12\;(\mathrm{mod}\,20)}
	+ \mathbf 1_{\,m\equiv 0\;(\mathrm{mod}\,12)}
	\in\{0,1,2\}.
	\]
	In particular, $a_m^{(F_4^{(1)})}=2$ if and only if
	\[
	m\equiv 12 \pmod{60}\quad\text{or}\quad m\equiv 48 \pmod{60}.
	\]
\end{itemize}

\begin{definition}\label{def:cp}
Let $\mathcal A_X:=\{\, (m,t)\mid m\in P_X,\ 1\le t\le a_m^{(X)}\,\}$
denote the set of \emph{allowable colored parts}. 
A \emph{$P_X$-colored partition} of an integer $p$ is a finite multiset $M$ of elements of $\mathcal A_X$ (with repetitions allowed) such that $	\sum_{(m,t)\in M} m = p$, where each occurrence of a part $(m,t)$ contributes $m$ to the total sum. Let $c(p)$ denote the number of $P_X$-colored partitions of $p$.
\end{definition}

\begin{remark}\label{rem:cp-order}
	In Definition~\ref{def:cp}, a $P_X$-colored partition is a finite multiset, so it carries no
	intrinsic ordering of its parts. When convenient, one may choose any fixed total order on
	$\mathcal A_X$ and list the parts in weakly decreasing order; the resulting sequence is merely a
	presentation of the same multiset.
\end{remark}

If $X\in\{D_4^{(3)},E_6^{(2)},E_6^{(1)},E_7^{(1)},E_8^{(1)}\}$, then $a_m^{(X)}\in\{0,1\}$ for all
$m$, so $c(p)$ is simply the number of (ordinary) partitions of $p$ whose parts lie in $P_X$.
Consequently,
\begin{equation}\label{eq:disjoint-cases-cp}
	\sum_{p\ge 0}c(p)\,q^p
	=\prod_{m\ge 1}(1-q^m)^{-a_m^{(X)}}
	=\prod_{r\in P_X}\frac{1}{1-q^r}.
\end{equation}

If $X\in\{G_2^{(1)},F_4^{(1)}\}$, then overlaps occur and $c(p)$ counts $P_X$-colored partitions,
i.e.\ a part of size $m$ is available in $a_m^{(X)}$ distinguishable colors. The generating
function is
\begin{equation}\label{eq:overlapping-cases}
	\sum_{p\ge 0}c(p)\,q^p
	=\prod_{m\ge 1}(1-q^m)^{-a_m^{(X)}}
	=\prod_{t=1}^{T_X}\ \prod_{\substack{k\ge 1\\ k\equiv r_t\;(\mathrm{mod}\,N_t)}} \frac{1}{1-q^k}
	=\prod_{r\in P_X}\frac{1}{1-q^r},
\end{equation}
where in the last expression $P_X$ is understood as a multiset (each congruence condition contributes
one copy of every integer it contains).

\begin{definition}\label{def:dp}
	For $p\in\mathbb Z_{\ge 0}$, define
	\[
	d(p):=\#\left\{\pi\in\mathcal P'_\phi\ \middle|\ \sum_{k=1}^{\ell(\pi)}\pi_k=p\right\}.
	\]
\end{definition}
Then, we obtain 
\begin{equation}\label{eq:dp}
	\sum_{p\ge 0} d(p)\,q^p=\sum_{\pi\in\mathcal P'_\phi} q^{\sum_{k=1}^{\ell(\pi)} \pi_k}.
\end{equation}

\begin{theorem}[Partition identity]\label{thm:partition-identity}
	With $c(p)$ and $d(p)$ as in Definitions~\ref{def:cp}--\ref{def:dp}, we have
	\[
	c(p)=d(p)\qquad\text{for all } p\in\mathbb Z_{\ge 0}.
	\]
\end{theorem}

\begin{proof}
	Recall the principal specialization of both sides of \eqref{eq:coloredKMN}:
	\begin{equation}\label{eq:DK20-ps-correct}
		\mathrm F_1\!\left(\sum_{\lambda\in\mathcal P_\phi^{\gg}} C(\lambda)\,e^{-\delta|\lambda|}\right)
		=
		\frac{\mathrm F_1\!\bigl(e^{-\Lambda_0}\,\mathrm{ch}\,L(\Lambda_0)\bigr)}
		{(q^{\operatorname{ht}(\delta)};q^{\operatorname{ht}(\delta)})_\infty}.
	\end{equation}
	
	By Definition~\ref{def:dp}, \eqref{eq:F1left} and \eqref{eq:dp}, we have
	\begin{equation}\label{eq:LHS-dp-correct}
		\mathrm{F}_1\!\left(\sum_{\lambda\in\mathcal{P}_\phi^{\gg}} C(\lambda)e^{-\delta|\lambda|}\right)
		=\sum_{\pi\in\mathcal P'_\phi} q^{\sum_{k=1}^{\ell(\pi)}\pi_k}
		=\sum_{p\ge 0} d(p)\,q^p.
	\end{equation}
	
	From the congruence data $\{(r_t,N_t)\}_{t=1}^{T_X}$ in Table~\ref{tab:PX-congruence-data},
	together with Definition~\ref{def:cp} and \eqref{eq:disjoint-cases-cp}--\eqref{eq:overlapping-cases},
	we obtain
	\begin{equation}\label{eq:RHS-cp-correct}
		\frac{\mathrm F_1\!\bigl(e^{-\Lambda_0}\,\mathrm{ch}\,L(\Lambda_0)\bigr)}
		{(q^{\operatorname{ht}(\delta)};q^{\operatorname{ht}(\delta)})_\infty}
		=
		\prod_{t=1}^{T_X}\ \prod_{\substack{k\ge 1\\ k\equiv r_t\;(\mathrm{mod}\,N_t)}} \frac{1}{1-q^k}
		=
		\prod_{m\ge 1}(1-q^m)^{-a_m^{(X)}}
		=\sum_{p\ge 0} c(p)\,q^p.
	\end{equation}
	
	Combining \eqref{eq:DK20-ps-correct}, \eqref{eq:LHS-dp-correct}, and \eqref{eq:RHS-cp-correct} yields
	$\sum_{p\ge 0} d(p)\,q^p=\sum_{p\ge 0} c(p)\,q^p$. Therefore $c(p)=d(p)$ for all
	$p\in\mathbb Z_{\ge 0}$.
\end{proof}

\section{Computational coefficient check (example: $p\le 60$)}\label{sec:coeff-check-p60}

As a reproducibility check, we expanded both sides of the generating-function identity
\[
\sum_{p\ge 0} c(p)q^p=\sum_{p\ge 0} d(p)q^p
\]
up to order $q^{60}$ for each exceptional affine type $X$ considered in this paper.
Here $c(p)$ is obtained by expanding the Euler-type product
$\prod_{m\ge 1}(1-q^m)^{-a_m^{(X)}}$, and $d(p)$ is obtained by enumerating the modified
grounded colored partitions $\pi\in\mathcal P'_\phi$ satisfying the congruence, initial,
and difference conditions encoded by $(\mathrm{CCon}_X,\mathrm{ICon}_X,M_X)$.
In all seven cases the coefficients agree for $1\le p\le 60$.

\begin{table}[H]
	\centering
	\renewcommand{\arraystretch}{1.15}
	\setlength{\tabcolsep}{4pt}
	\caption{Coefficient check for $c(p)=d(p)$ for $1\le p\le 60$}
	\label{tab:coeff-check-p60}
	\begin{tabular}{|c|p{0.9\textwidth}|}
		\hline
		Type $X$ & $c(p)=d(p)$ for $p=1,2,\dots,60$ \\
		\hline
		$G_2^{(1)}$ &
		$1,1,1,1,2,4,5,5,6,7,10,15,18,20,23,27,35,47,56,63,73,85,105,133,156,177,203,235,282,343,$\\
		&$399,452,516,593,698,829,954,1079,1225,1398,1622,1892,2161,2436,2753,3123,3583,4126,4680,$\\
		&$5258,5914,6672,7588,8650,9755,10920,12232,13732,15506,17537$ \\
		\hline
		$D_4^{(3)}$ &
		$1,1,1,2,3,3,4,6,7,8,10,14,17,19,23,30,36,41,49,61,72,82,97,119,139,158,184,220,256,291,$\\
		&$337,397,457,518,596,695,796,899,1027,1186,1351,1523,1731,1982,2246,2524,2856,3252,3669,$\\
		&$4111,4630,5240,5891,6584,7389,8322,9319,10388,11618,13032$ \\
		\hline
		$E_6^{(2)}$ &
		$1,1,1,1,2,2,3,3,4,5,6,7,8,10,11,13,15,19,22,25,28,32,38,43,50,56,65,74,84,95,107,122,136,$\\
		&$154,173,198,222,248,276,308,347,386,432,479,536,596,662,734,813,903,996,1103,1218,1352,$\\
		&$1492,1643,1807,1988,2193,2409$ \\
		\hline
		$F_4^{(1)}$ &
		$1,1,1,1,2,2,3,4,4,5,6,9,11,12,14,16,20,23,28,33,37,43,50,62,72,81,92,105,123,140,162,186,$\\
		&$209,237,270,314,357,400,450,507,576,648,733,825,921,1031,1157,1310,1467,1632,1817,2025,$\\
		&$2265,2521,2812,3129,3466,3843,4266,4754$ \\
		\hline
		$E_6^{(1)}$ &
		$1,1,1,2,3,3,4,6,7,8,10,14,17,19,23,30,36,41,49,61,72,82,97,119,139,158,184,220,256,291,$\\
		&$337,397,457,518,596,695,796,899,1027,1186,1351,1523,1731,1982,2246,2524,2856,3252,3669,$\\
		&$4111,4630,5240,5891,6584,7389,8322,9319,10388,11618,13032$ \\
		\hline
		$E_7^{(1)}$ &
		$1,1,1,1,2,2,3,3,4,5,6,7,8,10,11,13,15,19,22,25,28,32,38,43,50,56,65,74,84,95,107,122,136,$\\
		&$154,173,198,222,248,276,308,347,386,432,479,536,596,662,734,813,903,996,1103,1218,1352,$\\
		&$1492,1643,1807,1988,2193,2409$ \\
		\hline
		$E_8^{(1)}$ &
		$1,1,1,1,1,1,2,2,2,2,3,3,4,5,5,5,6,7,8,9,10,11,12,14,15,17,18,20,22,26,29,31,34,37,40,44,$\\
		&$50,54,58,63,70,76,84,92,99,106,116,127,138,150,162,175,189,206,222,240,258,278,300,328$ \\
		\hline
	\end{tabular}
\end{table}

\medskip
\begin{remark}
The Euler-type products coincide for the pairs $(D_4^{(3)},E_6^{(1)})$ and
$(E_6^{(2)},E_7^{(1)})$, since
\[
0\ (\mathrm{mod}\ 4)=\{0,4,8\}\ (\mathrm{mod}\ 12)
\qquad\text{and}\qquad
0\ (\mathrm{mod}\ 9)=\{0,9\}\ (\mathrm{mod}\ 18).
\]
\end{remark}

\appendix 

\section{The congruence conditions}  
\label{sec:congruence-condition}

\begin{table}[H]
	\centering
\scriptsize
\setlength{\tabcolsep}{2pt}
\renewcommand{\arraystretch}{1.5}
\resizebox{\textwidth}{!}{%
	\begin{minipage}[b]{0.485\textwidth} 
		\centering
		\renewcommand{\arraystretch}{1.5}	
		\captionof{table}{$\mathrm{CCon}_{G_2^{(1)}}$}
\begin{tabular}{|c|l|}
			\hline
			$\mathrm{pr}(\pi_k^c)$ & The color $c$ \\
			\hline
			$0\mod 6$ & $\phi$, $r_1$, $r_2$ \\
			\hline
			$1\mod 6$ & $\overline{(1,0)}$, $\overline{(0,1)}$, $(2,3)$ \\
			\hline
			$2\mod 6$ & $(1,3)$, $\overline{(1,1)}$ \\
			\hline
			$3\mod 6$ & $(1,2)$, $\overline{(1,2)}$ \\
			\hline
			$4\mod 6$ & $(1,1)$, $\overline{(1,3)}$ \\
			\hline
			$5\mod 6$ & $(1,0)$, $(0,1)$, $\overline{(2,3)}$ \\
			\hline
		\end{tabular}
	\end{minipage}
	\hspace{-2.4cm} 
	\begin{minipage}[b]{0.485\textwidth} 
		\centering
		\renewcommand{\arraystretch}{1.5}
				\captionof{table}{$\mathrm{CCon}_{D_4^{(3)}}$}
		\begin{tabular}{|c|l|}
			\hline
			$\mathrm{pr}(\pi_k^c)$ & The color $c$ \\
			\hline
			$0\mod 4$ & $\phi$, $r_1$ \\
			\hline
			$1\mod 4$ & $(21)$, $\overline{(10)}$ \\
			\hline
			$2\mod 4$ & $(11)$, $\overline{(11)}$ \\
			\hline
			$3\mod 4$ & $(10)$, $\overline{(21)}$ \\
			\hline
		\end{tabular}

	\end{minipage}
}	
\end{table}

\vskip -0.4cm

\begin{table}[H]
	\centering
\scriptsize
\setlength{\tabcolsep}{2pt}
\renewcommand{\arraystretch}{1.5}
\resizebox{\textwidth}{!}{%
	\begin{minipage}[b]{0.485\textwidth} 
		\centering
		\renewcommand{\arraystretch}{1.5}
		\captionof{table}{$\mathrm{CCon}_{E_6^{(2)}}$}
			\begin{tabular}{|c|l|}
			\hline
			$\mathrm{pr}(\pi_k^c)$ & The color $c$ \\
			\hline
			$0\mod 9$ & $\phi$, $r_1$, $r_2$ \\
			\hline
			$1\mod 9$ & $\overline{(1000)}$, $\overline{(0100)}$, $(2321)$ \\
			\hline
			$2\mod 9$ & $\overline{(0110)}$, $\overline{(1100)}$, $(1321)$ \\
			\hline
			$3\mod 9$ & $\overline{(0111)}$, $\overline{(1110)}$, $(1221)$ \\
			\hline
			$4\mod 9$ & $\overline{(1210)}$, $\overline{(1111)}$, $(1211)$ \\
			\hline
			$5\mod 9$ & $\overline{(1211)}$, $(1111)$, $(1210)$ \\
			\hline
			$6\mod 9$ & $\overline{(1221)}$, $(0111)$, $(1110)$ \\
			\hline
			$7\mod 9$ & $\overline{(1321)}$, $(0110)$, $(1100)$ \\
			\hline
			$8\mod 9$ & $\overline{(2321)}$, $(1000)$, $(0100)$ \\
			\hline
		\end{tabular}
	
	\end{minipage}
	\hspace{-0.45cm} 
	\begin{minipage}[b]{0.485\textwidth} 
		\centering
		\renewcommand{\arraystretch}{1.5}
			\captionof{table}{$\mathrm{CCon}_{F_4^{(1)}}$}
		\begin{tabular}{|c|l|}
			\hline
			$\mathrm{pr}(\pi_k^c)$ & The color $c$ \\
			\hline
			$0 \bmod 12$  & $\phi,\ r_1,\ r_2,\ r_3,\ r_4$ \\
			\hline
			$1 \bmod 12$  & $\overline{(1000)},\ \overline{(0100)},\ \overline{(0010)},\ \overline{(0001)},\ (2342)$ \\
			\hline
			$2 \bmod 12$  & $\overline{(1100)},\ \overline{(0110)},\ \overline{(0011)},\ (1342)$ \\
			\hline
			$3 \bmod 12$  & $\overline{(1110)},\ \overline{(0120)},\ \overline{(0111)},\ (1242)$ \\
			\hline
			$4 \bmod 12$  & $\overline{(1120)},\ \overline{(1111)},\ \overline{(0121)},\ (1232)$ \\
			\hline
			$5 \bmod 12$  & $\overline{(1220)},\ \overline{(1121)},\ \overline{(0122)},\ (1231),\ (1222)$ \\
			\hline
			$6 \bmod 12$  & $\overline{(1221)},\ \overline{(1122)},\ (1221),\ (1122)$ \\
			\hline
			$7 \bmod 12$  & $\overline{(1231)},\ \overline{(1222)},\ (1220),\ (1121),\ (0122)$ \\
			\hline
			$8 \bmod 12$  & $\overline{(1232)},\ (1120),\ (1111),\ (0121)$ \\
			\hline
			$9 \bmod 12$  & $\overline{(1242)},\ (1110),\ (0120),\ (0111)$ \\
			\hline
			$10 \bmod 12$ & $\overline{(1342)},\ (1100),\ (0110),\ (0011)$ \\
			\hline
			$11 \bmod 12$ & $\overline{(2342)},\ (1000),\ (0100),\ (0010),\ (0001)$ \\
			\hline
		\end{tabular}
	
	\end{minipage}
}	
\end{table}

\vskip -0.2cm

\begin{table}[H]
		\caption{$\mathrm{CCon}_{E_6^{(1)}}$}
	\centering
\scriptsize
\setlength{\tabcolsep}{2pt}
\renewcommand{\arraystretch}{1.5}
	\begin{tabular}{|c|l|}
		\hline
		$\mathrm{pr}(\pi_k^c)$ & The color $c$ \\
		\hline
		$0 \bmod 12$ 
		& $\phi$, $r_1$, $r_2$, $r_3$, $r_4$, $r_5$, $r_6$ \\
		\hline
		
		$1 \bmod 12$ 
		& $\overline{(100000)}$, $\overline{(010000)}$, $\overline{(001000)}$, $\overline{(000100)}$, $\overline{(000010)}$, $\overline{(000001)}$, $(122321)$ \\ 
		\hline
		
		$2 \bmod 12$ 
		& $\overline{(101000)}$, $\overline{(010100)}$, $\overline{(001100)}$, $\overline{(000110)}$, $\overline{(000011)}$, $(112321)$ \\ 
		\hline
		
		$3 \bmod 12$ 
		& $\overline{(101100)}$, $\overline{(011100)}$, $\overline{(010110)}$, $\overline{(001110)}$, $\overline{(000111)}$, $(112221)$ \\ 
		\hline
		
		$4 \bmod 12$ 
		& $\overline{(111100)}$, $\overline{(101110)}$, $\overline{(011110)}$, $\overline{(010111)}$, $\overline{(001111)}$, $(112211)$, $(111221)$ \\ 
		\hline
		
		$5 \bmod 12$ 
		& $\overline{(111110)}$, $\overline{(101111)}$, $\overline{(011210)}$, $\overline{(011111)}$, 
		$(112210)$, $(111211)$, $(011221)$ \\ 
		\hline
		
		$6 \bmod 12$ 
		& $\overline{(111210)}$, $\overline{(111111)}$, $\overline{(011211)}$, 
		$(111210)$, $(111111)$, $(011211)$ \\ 
		\hline
		
		$7 \bmod 12$ 
		& $\overline{(112210)}$, $\overline{(111211)}$, $\overline{(011221)}$, 
		$(111110)$, $(101111)$, $(011210)$, $(011111)$ \\ 
		\hline
		
		$8 \bmod 12$ 
		& $\overline{(112211)}$, $\overline{(111221)}$, 
		$(111100)$, $(101110)$, $(011110)$, $(010111)$, $(001111)$ \\ 
		\hline
		
		$9 \bmod 12$ 
		& $\overline{(112221)}$, 
		$(101100)$, $(011100)$, $(010110)$, $(001110)$, $(000111)$ \\ 
		\hline
		
		$10 \bmod 12$ 
		& $\overline{(112321)}$, 
		$(101000)$, $(010100)$, $(001100)$, $(000110)$, $(000011)$ \\ 
		\hline
		
		$11 \bmod 12$ 
		& $\overline{(122321)}$, 
		$(100000)$, $(010000)$, $(001000)$, $(000100)$, $(000010)$, $(000001)$ \\ 
		\hline
	\end{tabular}
\end{table}

\begin{table}[H]
		\caption{$\mathrm{CCon}_{E_7^{(1)}}$}
	\centering
\scriptsize
\setlength{\tabcolsep}{2pt}
\renewcommand{\arraystretch}{1.5}
	\begin{tabular}{|c|l|}
		\hline
		$\mathrm{pr}(\pi_k^c)$ & The color $c$ \\
		\hline
		$0 \bmod 18$
		& $\phi,\ r_1,\ r_2,\ r_3,\ r_4,\ r_5,\ r_6,\ r_7$ \\
		\hline
		$1 \bmod 18$
		& $\overline{(1000000)},\ \overline{(0100000)},\ \overline{(0010000)},\ \overline{(0001000)},\ \overline{(0000100)},\ \overline{(0000010)},\ \overline{(0000001)},\ (2234321)$ \\
		\hline
		$2 \bmod 18$
		& $\overline{(1010000)},\ \overline{(0101000)},\ \overline{(0011000)},\ \overline{(0001100)},\ \overline{(0000110)},\ \overline{(0000011)},\ (1234321)$ \\
		\hline
		$3 \bmod 18$
		& $\overline{(1011000)},\ \overline{(0111000)},\ \overline{(0101100)},\ \overline{(0011100)},\ \overline{(0001110)},\ \overline{(0000111)},\ (1224321)$ \\
		\hline
		$4 \bmod 18$
		& $\overline{(1111000)},\ \overline{(1011100)},\ \overline{(0111100)},\ \overline{(0101110)},\ \overline{(0011110)},\ \overline{(0001111)},\ (1223321)$ \\
		\hline
		$5 \bmod 18$
		& $\overline{(1111100)},\ \overline{(1011110)},\ \overline{(0112100)},\ \overline{(0111110)},\ \overline{(0101111)},\ \overline{(0011111)},\ (1223221),\ (1123321)$ \\
		\hline
		$6 \bmod 18$
		& $\overline{(1112100)},\ \overline{(1111110)},\ \overline{(1011111)},\ \overline{(0112110)},\ \overline{(0111111)},\ (1223211),\ (1123221)$ \\
\hline
	\end{tabular}
\end{table}

\begin{table}[H]
	\caption{$\mathrm{CCon}_{E_8^{(1)}}$}	
	\centering
	
\scriptsize
\setlength{\tabcolsep}{1.7pt}
\renewcommand{\arraystretch}{1.2}

	\begin{tabular}{|c|l|}
		\hline
		$\mathrm{pr}(\pi_k^c)$ & The color $c$ \\
		\hline
		$0 \bmod 30$
		& $\phi,\ r_1,\ r_2,\ r_3,\ r_4,\ r_5,\ r_6,\ r_7,\ r_8$ \\
		\hline
		$1 \bmod 30$
		& $(23465432),\ \overline{(00000001)},\ \overline{(00000010)},\ \overline{(00000100)},\ \overline{(00001000)},\ \overline{(00010000)},\ \overline{(00100000)},\ \overline{(01000000)},\ \overline{(10000000)}$ \\
		\hline
		$2 \bmod 30$
		& $(23465431),\ \overline{(00000011)},\ \overline{(00000110)},\ \overline{(00001100)},\ \overline{(00011000)},\ \overline{(00110000)},\ \overline{(01010000)},\ \overline{(10100000)}$ \\
		\hline
		$3 \bmod 30$
		& $(23465421),\ \overline{(00000111)},\ \overline{(00001110)},\ \overline{(00011100)},\ \overline{(00111000)},\ \overline{(01011000)},\ \overline{(01110000)},\ \overline{(10110000)}$ \\
		\hline
		$4 \bmod 30$
		& $(23465321),\ \overline{(00001111)},\ \overline{(00011110)},\ \overline{(00111100)},\ \overline{(01011100)},\ \overline{(01111000)},\ \overline{(10111000)},\ \overline{(11110000)}$ \\
		\hline
		$5 \bmod 30$
		& $(23464321),\ \overline{(00011111)},\ \overline{(00111110)},\ \overline{(01011110)},\ \overline{(01111100)},\ \overline{(01121000)},\ \overline{(10111100)},\ \overline{(11111000)}$ \\
		\hline
		$6 \bmod 30$
		& $(23454321),\ \overline{(00111111)},\ \overline{(01011111)},\ \overline{(01111110)},\ \overline{(01121100)},\ \overline{(10111110)},\ \overline{(11111100)},\ \overline{(11121000)}$ \\
\hline
	\end{tabular}
\end{table}

\section{The energy matrices}  
\label{sec:energy-matrix}

\begin{table}[H]
	\centering
	\begin{minipage}[t]{0.5\textwidth}
	\centering	
		\caption{Difference matrix $M_{D_4^{(3)}}$}
		\label{tab:energy-matrix-D43-reordered}
	\scriptsize
\setlength{\tabcolsep}{3.2pt}
\renewcommand{\arraystretch}{1.4}	
\begin{tabular}{|c|c|c|c|c|c|c|c|c|}
	\hline
	& \rotatebox{90}{$\phi$}
	& \rotatebox{90}{$(21)$}
	& \rotatebox{90}{$(11)$}
	& \rotatebox{90}{$(10)$}
	& \rotatebox{90}{$r_1$}
	& \rotatebox{90}{$\overline{(10)}$}
	& \rotatebox{90}{$\overline{(11)}$}
	& \rotatebox{90}{$\overline{(21)}$}
	\\
	\hline
	$\phi$ & 0 & 1 & 2 & 3 & 4 & 5 & 6 & 7 \\
	\hline
	$(21)$ & 7 & 8 & 5 & 6 & 3 & 4 & 5 & 6 \\
	\hline
	$(11)$ & 6 & 7 & 8 & 5 & 6 & 7 & 4 & 5 \\
	\hline
	$(10)$ & 5 & 6 & 7 & 8 & 5 & 6 & 7 & 4 \\
	\hline
	$r_1$ & 4 & 5 & 6 & 7 & 4 & 5 & 6 & 3 \\
	\hline
	$\overline{(10)}$ & 3 & 4 & 5 & 6 & 7 & 8 & 5 & 6 \\
	\hline
	$\overline{(11)}$ & 2 & 3 & 4 & 5 & 6 & 7 & 8 & 5 \\
	\hline
	$\overline{(21)}$ & 1 & 2 & 3 & 4 & 5 & 6 & 7 & 8 \\
	\hline
\end{tabular}
\end{minipage}
\hspace{-1.6cm}
\begin{minipage}[t]{0.48\textwidth}
	\centering	
\caption{Difference matrix $M_{G_2^{(1)}}$}
\label{tab:energy-matrix-G21}
	\scriptsize
\setlength{\tabcolsep}{1.5pt}
\renewcommand{\arraystretch}{1.2}
\begin{tabular}{|c|c|c|c|c|c|c|c|c|c|c|c|c|c|c|c|}
	\hline
	& \rotatebox{90}{$\phi$}
	& \rotatebox{90}{$(23)$}
	& \rotatebox{90}{$(13)$}
	& \rotatebox{90}{$(12)$}
	& \rotatebox{90}{$(11)$}
	& \rotatebox{90}{$(10)$}
	& \rotatebox{90}{$(01)$}
	& \rotatebox{90}{$r_1$}
	& \rotatebox{90}{$r_2$}
	& \rotatebox{90}{$\overline{(01)}$}
	& \rotatebox{90}{$\overline{(10)}$}
	& \rotatebox{90}{$\overline{(11)}$}
	& \rotatebox{90}{$\overline{(12)}$}
	& \rotatebox{90}{$\overline{(13)}$}
	& \rotatebox{90}{$\overline{(23)}$}
	\\
	\hline
	$\phi$ & 0 & 1 & 2 & 3 & 4 & 5 & 5 & 6 & 6 & 7 & 7 & 8 & 9 & 10 & 11 \\
	\hline
	$(23)$ & 11 & 12 & 7 & 8 & 9 & 10 & 4 & 5 & 5 & 6 & 6 & 7 & 8 & 9 & 10 \\
	\hline
	$(13)$ & 10 & 11 & 12 & 7 & 8 & 9 & 9 & 10 & 4 & 5 & 11 & 6 & 7 & 8 & 9 \\
	\hline
	$(12)$ & 9 & 10 & 11 & 6 & 7 & 8 & 8 & 9 & 3 & 4 & 10 & 5 & 6 & 7 & 8 \\
	\hline
	$(11)$ & 8 & 9 & 10 & 5 & 6 & 7 & 7 & 8 & 2 & 3 & 9 & 4 & 5 & 6 & 7 \\
	\hline
	$(10)$ & 7 & 8 & 9 & 10 & 11 & 12 & 6 & 7 & 7 & 8 & 8 & 9 & 10 & 11 & 6 \\
	\hline
	$(01)$ & 7 & 8 & 9 & 4 & 5 & 6 & 6 & 7 & 1 & 2 & 8 & 3 & 4 & 5 & 6 \\
	\hline
	$r_1$ & 6 & 7 & 8 & 9 & 10 & 11 & 5 & 6 & 6 & 7 & 7 & 8 & 9 & 10 & 5 \\
	\hline
	$r_2$ & 6 & 7 & 8 & 3 & 4 & 5 & 5 & 6 & 0 & 1 & 7 & 2 & 3 & 4 & 5 \\
	\hline
	$\overline{(01)}$ & 5 & 6 & 7 & 8 & 9 & 10 & 4 & 5 & 5 & 6 & 6 & 7 & 8 & 9 & 4 \\
	\hline
	$\overline{(10)}$ & 5 & 6 & 7 & 8 & 9 & 10 & 10 & 11 & 5 & 6 & 12 & 7 & 8 & 9 & 10 \\
	\hline
	$\overline{(11)}$ & 4 & 5 & 6 & 7 & 8 & 9 & 9 & 10 & 4 & 5 & 11 & 6 & 7 & 8 & 9 \\
	\hline
	$\overline{(12)}$ & 3 & 4 & 5 & 6 & 7 & 8 & 8 & 9 & 3 & 4 & 10 & 5 & 6 & 7 & 8 \\
	\hline
	$\overline{(13)}$ & 2 & 3 & 4 & 5 & 6 & 7 & 7 & 8 & 8 & 9 & 9 & 10 & 11 & 12 & 7 \\
	\hline
	$\overline{(23)}$ & 1 & 2 & 3 & 4 & 5 & 6 & 6 & 7 & 7 & 8 & 8 & 9 & 10 & 11 & 12 \\
	\hline
\end{tabular}	
\end{minipage}
\end{table}

\begin{table}[H]
	\centering
	\caption{Difference matrix $M_{E_6^{(2)}}$}	
\scriptsize
\setlength{\tabcolsep}{1.8pt}
\renewcommand{\arraystretch}{1.3}	
		\begin{tabular}{|c|c|c|c|c|c|c|c|c|c|c|c|c|c|c|c|c|c|c|c|c|c|c|c|c|c|c|c|}
			\hline
			& \rotatebox{90}{$\phi$}
			& \rotatebox{90}{$(2321)$}
			& \rotatebox{90}{$(1321)$}
			& \rotatebox{90}{$(1221)$}
			& \rotatebox{90}{$(1211)$}
			& \rotatebox{90}{$(1210)$}
			& \rotatebox{90}{$(1111)$}
			& \rotatebox{90}{$(1110)$}
			& \rotatebox{90}{$(1100)$}
			& \rotatebox{90}{$(1000)$}
			& \rotatebox{90}{$(0111)$}
			& \rotatebox{90}{$(0110)$}
			& \rotatebox{90}{$(0100)$}
			& \rotatebox{90}{$r_1$}
			& \rotatebox{90}{$r_2$}
			& \rotatebox{90}{$\overline{(0100)}$}
			& \rotatebox{90}{$\overline{(0110)}$}
			& \rotatebox{90}{$\overline{(0111)}$}
			& \rotatebox{90}{$\overline{(1000)}$}
			& \rotatebox{90}{$\overline{(1100)}$}
			& \rotatebox{90}{$\overline{(1110)}$}
			& \rotatebox{90}{$\overline{(1111)}$}
			& \rotatebox{90}{$\overline{(1210)}$}
			& \rotatebox{90}{$\overline{(1211)}$}
			& \rotatebox{90}{$\overline{(1221)}$}
			& \rotatebox{90}{$\overline{(1321)}$}
			& \rotatebox{90}{$\overline{(2321)}$} 
			\\
			\hline
			$\phi$ & 0 & 1 & 2 & 3 & 4 & 5 & 5 & 6 & 7 & 8 & 6 & 7 & 8 & 9 & 9 & 10 & 11 & 12 & 10 & 11 & 12 & 13 & 13 & 14 & 15 & 16 & 17 \\
			\hline
			$(2321)$ & 17 & 18 & 10 & 11 & 12 & 13 & 13 & 14 & 15 & 16 & 5 & 6 & 7 & 8 & 8 & 9 & 10 & 11 & 9 & 10 & 11 & 12 & 12 & 13 & 14 & 15 & 16 \\
			\hline
			$(1321)$ & 16 & 17 & 18 & 10 & 11 & 12 & 12 & 13 & 14 & 15 & 13 & 14 & 15 & 16 & 7 & 8 & 9 & 10 & 17 & 9 & 10 & 11 & 11 & 12 & 13 & 14 & 15 \\
			\hline
			$(1221)$ & 15 & 16 & 17 & 18 & 10 & 11 & 11 & 12 & 13 & 14 & 12 & 13 & 14 & 15 & 15 & 16 & 8 & 9 & 16 & 17 & 9 & 10 & 10 & 11 & 12 & 13 & 14 \\
			\hline
			$(1211)$ & 14 & 15 & 16 & 17 & 18 & 10 & 10 & 11 & 12 & 13 & 11 & 12 & 13 & 14 & 14 & 15 & 16 & 8 & 15 & 16 & 17 & 9 & 9 & 10 & 11 & 12 & 13 \\
			\hline
			$(1210)$ & 13 & 14 & 15 & 16 & 17 & 18 & 9 & 10 & 11 & 12 & 10 & 11 & 12 & 13 & 13 & 14 & 15 & 16 & 14 & 15 & 16 & 17 & 8 & 9 & 10 & 11 & 12 \\
\hline
		\end{tabular}
\end{table}

	\clearpage
\newgeometry{left=0.2cm,right=0.2cm,top=1.2cm,bottom=0.2cm}		
\begin{landscape}
	
	\centering
	\fontsize{8.3pt}{0.1pt}\selectfont
	\setlength{\tabcolsep}{0pt}
	\renewcommand{\arraystretch}{0.6}
	
	\captionof{table}{Difference matrix $M_{F_4^{(1)}}$}
\begin{longtable}
		{|c|c|c|c|c|c|c|c|c|c|c|c|c|c|c|c|c|c|c|c|c|c|c|c|c|c|c|c|c|c|c|c|c|c|c|c|c|c|c|c|c|c|c|c|c|c|c|c|c|c|c|c|c|c|}
		\hline
		& \rotatebox{90}{$\phi$}
		& \rotatebox{90}{$(2342)$}
		& \rotatebox{90}{$(1342)$}
		& \rotatebox{90}{$(1242)$}
		& \rotatebox{90}{$(1232)$}
		& \rotatebox{90}{$(1231)$}
		& \rotatebox{90}{$(1222)$}
		& \rotatebox{90}{$(1221)$}
		& \rotatebox{90}{$(1220)$}
		& \rotatebox{90}{$(1122)$}
		& \rotatebox{90}{$(1121)$}
		& \rotatebox{90}{$(1120)$}
		& \rotatebox{90}{$(1111)$}
		& \rotatebox{90}{$(1110)$}
		& \rotatebox{90}{$(1100)$}
		& \rotatebox{90}{$(1000)$}
		& \rotatebox{90}{$(0122)$}
		& \rotatebox{90}{$(0121)$}
		& \rotatebox{90}{$(0120)$}
		& \rotatebox{90}{$(0111)$}
		& \rotatebox{90}{$(0110)$}
		& \rotatebox{90}{$(0100)$}
		& \rotatebox{90}{$(0011)$}
		& \rotatebox{90}{$(0010)$}
		& \rotatebox{90}{$(0001)$}
		& \rotatebox{90}{$r_1$}
		& \rotatebox{90}{$r_2$}
		& \rotatebox{90}{$r_3$}
		& \rotatebox{90}{$r_4$}
		& \rotatebox{90}{$\overline{(0001)}$}
		& \rotatebox{90}{$\overline{(0010)}$}
		& \rotatebox{90}{$\overline{(0011)}$}
		& \rotatebox{90}{$\overline{(0100)}$}
		& \rotatebox{90}{$\overline{(0110)}$}
		& \rotatebox{90}{$\overline{(0111)}$}
		& \rotatebox{90}{$\overline{(0120)}$}
		& \rotatebox{90}{$\overline{(0121)}$}
		& \rotatebox{90}{$\overline{(0122)}$}
		& \rotatebox{90}{$\overline{(1000)}$}
		& \rotatebox{90}{$\overline{(1100)}$}
		& \rotatebox{90}{$\overline{(1110)}$}
		& \rotatebox{90}{$\overline{(1111)}$}
		& \rotatebox{90}{$\overline{(1120)}$}
		& \rotatebox{90}{$\overline{(1121)}$}
		& \rotatebox{90}{$\overline{(1122)}$}
		& \rotatebox{90}{$\overline{(1220)}$}
		& \rotatebox{90}{$\overline{(1221)}$}
		& \rotatebox{90}{$\overline{(1222)}$}
		& \rotatebox{90}{$\overline{(1231)}$}
		& \rotatebox{90}{$\overline{(1232)}$}
		& \rotatebox{90}{$\overline{(1242)}$}
		& \rotatebox{90}{$\overline{(1342)}$}
		& \rotatebox{90}{$\overline{(2342)}$}
		\\
\hline
	\end{longtable}
\end{landscape}
\restoregeometry

	\clearpage
\newgeometry{left=0.2cm,right=0.2cm,top=1.2cm,bottom=0.2cm}		
\begin{landscape}
	\centering
	\fontsize{8pt}{8pt}\selectfont
	\setlength{\tabcolsep}{0pt}
	\renewcommand{\arraystretch}{0.6}
	
	\captionof{table}{Difference matrix $M_{E_6^{(1)}}$}
	%
\end{landscape}
\restoregeometry

	\clearpage
\newgeometry{left=0.2cm,right=0.2cm,top=1.2cm,bottom=0.2cm}		
\begin{landscape}
	\centering
	\fontsize{10pt}{8pt}\selectfont
	\setlength{\tabcolsep}{0pt}
	\renewcommand{\arraystretch}{0.6}
	
	\captionof{table}{Difference matrix $M_{E_7^{(1)}}$ (Part A)}
	\begin{longtable}{|c|*{71}{c|}}
		\hline
 &  \rotatebox{90}{$\phi$} &  \rotatebox{90}{$(2234321)$} &  \rotatebox{90}{$(1234321)$} &  \rotatebox{90}{$(1224321)$} &  \rotatebox{90}{$(1223321)$} &  \rotatebox{90}{$(1223221)$} &  \rotatebox{90}{$(1223211)$} &  \rotatebox{90}{$(1223210)$} &  \rotatebox{90}{$(1123321)$} &  \rotatebox{90}{$(1123221)$} &  \rotatebox{90}{$(1123211)$} &  \rotatebox{90}{$(1123210)$} &  \rotatebox{90}{$(1122221)$} &  \rotatebox{90}{$(1122211)$} &  \rotatebox{90}{$(1122210)$} &  \rotatebox{90}{$(1122111)$} &  \rotatebox{90}{$(1122110)$} &  \rotatebox{90}{$(1122100)$} &  \rotatebox{90}{$(1112221)$} &  \rotatebox{90}{$(1112211)$} &  \rotatebox{90}{$(1112210)$} &  \rotatebox{90}{$(1112111)$} &  \rotatebox{90}{$(1112110)$} &  \rotatebox{90}{$(1112100)$} &  \rotatebox{90}{$(1111111)$} &  \rotatebox{90}{$(1111110)$} &  \rotatebox{90}{$(1111100)$} &  \rotatebox{90}{$(1111000)$} &  \rotatebox{90}{$(1011111)$} &  \rotatebox{90}{$(1011110)$} &  \rotatebox{90}{$(1011100)$} &  \rotatebox{90}{$(1011000)$} &  \rotatebox{90}{$(1010000)$} &  \rotatebox{90}{$(1000000)$} &  \rotatebox{90}{$(0112221)$} &  \rotatebox{90}{$(0112211)$} &  \rotatebox{90}{$(0112210)$} &  \rotatebox{90}{$(0112111)$} &  \rotatebox{90}{$(0112110)$} &  \rotatebox{90}{$(0112100)$} &  \rotatebox{90}{$(0111111)$} &  \rotatebox{90}{$(0111110)$} &  \rotatebox{90}{$(0111100)$} &  \rotatebox{90}{$(0111000)$} &  \rotatebox{90}{$(0101111)$} &  \rotatebox{90}{$(0101110)$} &  \rotatebox{90}{$(0101100)$} &  \rotatebox{90}{$(0101000)$} &  \rotatebox{90}{$(0100000)$} &  \rotatebox{90}{$(0011111)$} &  \rotatebox{90}{$(0011110)$} &  \rotatebox{90}{$(0011100)$} &  \rotatebox{90}{$(0011000)$} &  \rotatebox{90}{$(0010000)$} &  \rotatebox{90}{$(0001111)$} &  \rotatebox{90}{$(0001110)$} &  \rotatebox{90}{$(0001100)$} &  \rotatebox{90}{$(0001000)$} &  \rotatebox{90}{$(0000111)$} &  \rotatebox{90}{$(0000110)$} &  \rotatebox{90}{$(0000100)$} &  \rotatebox{90}{$(0000011)$} &  \rotatebox{90}{$(0000010)$} &  \rotatebox{90}{$(0000001)$} &  \rotatebox{90}{$r_1$} &  \rotatebox{90}{$r_2$} &  \rotatebox{90}{$r_3$} &  \rotatebox{90}{$r_4$} &  \rotatebox{90}{$r_5$} &  \rotatebox{90}{$r_6$} &  \rotatebox{90}{$r_7$}\\
		\hline
		\endfirsthead
		
		\hline
 &  \rotatebox{90}{$\phi$} &  \rotatebox{90}{$(2234321)$} &  \rotatebox{90}{$(1234321)$} &  \rotatebox{90}{$(1224321)$} &  \rotatebox{90}{$(1223321)$} &  \rotatebox{90}{$(1223221)$} &  \rotatebox{90}{$(1223211)$} &  \rotatebox{90}{$(1223210)$} &  \rotatebox{90}{$(1123321)$} &  \rotatebox{90}{$(1123221)$} &  \rotatebox{90}{$(1123211)$} &  \rotatebox{90}{$(1123210)$} &  \rotatebox{90}{$(1122221)$} &  \rotatebox{90}{$(1122211)$} &  \rotatebox{90}{$(1122210)$} &  \rotatebox{90}{$(1122111)$} &  \rotatebox{90}{$(1122110)$} &  \rotatebox{90}{$(1122100)$} &  \rotatebox{90}{$(1112221)$} &  \rotatebox{90}{$(1112211)$} &  \rotatebox{90}{$(1112210)$} &  \rotatebox{90}{$(1112111)$} &  \rotatebox{90}{$(1112110)$} &  \rotatebox{90}{$(1112100)$} &  \rotatebox{90}{$(1111111)$} &  \rotatebox{90}{$(1111110)$} &  \rotatebox{90}{$(1111100)$} &  \rotatebox{90}{$(1111000)$} &  \rotatebox{90}{$(1011111)$} &  \rotatebox{90}{$(1011110)$} &  \rotatebox{90}{$(1011100)$} &  \rotatebox{90}{$(1011000)$} &  \rotatebox{90}{$(1010000)$} &  \rotatebox{90}{$(1000000)$} &  \rotatebox{90}{$(0112221)$} &  \rotatebox{90}{$(0112211)$} &  \rotatebox{90}{$(0112210)$} &  \rotatebox{90}{$(0112111)$} &  \rotatebox{90}{$(0112110)$} &  \rotatebox{90}{$(0112100)$} &  \rotatebox{90}{$(0111111)$} &  \rotatebox{90}{$(0111110)$} &  \rotatebox{90}{$(0111100)$} &  \rotatebox{90}{$(0111000)$} &  \rotatebox{90}{$(0101111)$} &  \rotatebox{90}{$(0101110)$} &  \rotatebox{90}{$(0101100)$} &  \rotatebox{90}{$(0101000)$} &  \rotatebox{90}{$(0100000)$} &  \rotatebox{90}{$(0011111)$} &  \rotatebox{90}{$(0011110)$} &  \rotatebox{90}{$(0011100)$} &  \rotatebox{90}{$(0011000)$} &  \rotatebox{90}{$(0010000)$} &  \rotatebox{90}{$(0001111)$} &  \rotatebox{90}{$(0001110)$} &  \rotatebox{90}{$(0001100)$} &  \rotatebox{90}{$(0001000)$} &  \rotatebox{90}{$(0000111)$} &  \rotatebox{90}{$(0000110)$} &  \rotatebox{90}{$(0000100)$} &  \rotatebox{90}{$(0000011)$} &  \rotatebox{90}{$(0000010)$} &  \rotatebox{90}{$(0000001)$} &  \rotatebox{90}{$r_1$} &  \rotatebox{90}{$r_2$} &  \rotatebox{90}{$r_3$} &  \rotatebox{90}{$r_4$} &  \rotatebox{90}{$r_5$} &  \rotatebox{90}{$r_6$} &  \rotatebox{90}{$r_7$}\\
		\hline
		\endhead
		
		\hline
		\endfoot
		\hline
		\endlastfoot

$\phi$ &  0 &  1 &  2 &  3 &  4 &  5 &  6 &  7 &  5 &  6 &  7 &  8 &  7 &  8 &  9 &  9 &  10 &  11 &  8 &  9 &  10 &  10 &  11 &  12 &  11 &  12 &  13 &  14 &  12 &  13 &  14 &  15 &  16 &  17 &  9 &  10 &  11 &  11 &  12 &  13 &  12 &  13 &  14 &  15 &  13 &  14 &  15 &  16 &  17 &  13 &  14 &  15 &  16 &  17 &  14 &  15 &  16 &  17 &  15 &  16 &  17 &  16 &  17 &  17 &  18 &  18 &  18 &  18 &  18 &  18 &  18\\
\hline
$(2234321)$ &  35 &  36 &  19 &  20 &  21 &  22 &  23 &  24 &  22 &  23 &  24 &  25 &  24 &  25 &  26 &  26 &  27 &  28 &  25 &  26 &  27 &  27 &  28 &  29 &  28 &  29 &  30 &  31 &  29 &  30 &  31 &  32 &  33 &  34 &  8 &  9 &  10 &  10 &  11 &  12 &  11 &  12 &  13 &  14 &  12 &  13 &  14 &  15 &  16 &  12 &  13 &  14 &  15 &  16 &  13 &  14 &  15 &  16 &  14 &  15 &  16 &  15 &  16 &  16 &  17 &  17 &  17 &  17 &  17 &  17 &  17\\
\hline
$(1234321)$ &  34 &  35 &  36 &  19 &  20 &  21 &  22 &  23 &  21 &  22 &  23 &  24 &  23 &  24 &  25 &  25 &  26 &  27 &  24 &  25 &  26 &  26 &  27 &  28 &  27 &  28 &  29 &  30 &  28 &  29 &  30 &  31 &  32 &  33 &  25 &  26 &  27 &  27 &  28 &  29 &  28 &  29 &  30 &  31 &  11 &  12 &  13 &  14 &  15 &  29 &  30 &  31 &  32 &  33 &  12 &  13 &  14 &  15 &  13 &  14 &  15 &  14 &  15 &  15 &  34 &  16 &  16 &  16 &  16 &  16 &  16\\
\hline
	\end{longtable}	
\end{landscape}
\restoregeometry

\clearpage
\newgeometry{left=0.2cm,right=0.2cm,top=1.2cm,bottom=0.2cm}
\begin{landscape}
	\centering
	\fontsize{9.25pt}{8pt}\selectfont
	\setlength{\tabcolsep}{0pt}
	\renewcommand{\arraystretch}{0.6}
	
	\addtocounter{table}{-1}
	\captionsetup{labelformat=empty}
	\captionof{table}{Difference matrix $M_{E_7^{(1)}}$ (Part B, continued)}
	\captionsetup{labelformat=default}
	
	\begin{longtable}{|c|*{63}{c|}}
		\hline
 &  \rotatebox{90}{$\overline{(0000001)}$} &  \rotatebox{90}{$\overline{(0000010)}$} &  \rotatebox{90}{$\overline{(0000011)}$} &  \rotatebox{90}{$\overline{(0000100)}$} &  \rotatebox{90}{$\overline{(0000110)}$} &  \rotatebox{90}{$\overline{(0000111)}$} &  \rotatebox{90}{$\overline{(0001000)}$} &  \rotatebox{90}{$\overline{(0001100)}$} &  \rotatebox{90}{$\overline{(0001110)}$} &  \rotatebox{90}{$\overline{(0001111)}$} &  \rotatebox{90}{$\overline{(0010000)}$} &  \rotatebox{90}{$\overline{(0011000)}$} &  \rotatebox{90}{$\overline{(0011100)}$} &  \rotatebox{90}{$\overline{(0011110)}$} &  \rotatebox{90}{$\overline{(0011111)}$} &  \rotatebox{90}{$\overline{(0100000)}$} &  \rotatebox{90}{$\overline{(0101000)}$} &  \rotatebox{90}{$\overline{(0101100)}$} &  \rotatebox{90}{$\overline{(0101110)}$} &  \rotatebox{90}{$\overline{(0101111)}$} &  \rotatebox{90}{$\overline{(0111000)}$} &  \rotatebox{90}{$\overline{(0111100)}$} &  \rotatebox{90}{$\overline{(0111110)}$} &  \rotatebox{90}{$\overline{(0111111)}$} &  \rotatebox{90}{$\overline{(0112100)}$} &  \rotatebox{90}{$\overline{(0112110)}$} &  \rotatebox{90}{$\overline{(0112111)}$} &  \rotatebox{90}{$\overline{(0112210)}$} &  \rotatebox{90}{$\overline{(0112211)}$} &  \rotatebox{90}{$\overline{(0112221)}$} &  \rotatebox{90}{$\overline{(1000000)}$} &  \rotatebox{90}{$\overline{(1010000)}$} &  \rotatebox{90}{$\overline{(1011000)}$} &  \rotatebox{90}{$\overline{(1011100)}$} &  \rotatebox{90}{$\overline{(1011110)}$} &  \rotatebox{90}{$\overline{(1011111)}$} &  \rotatebox{90}{$\overline{(1111000)}$} &  \rotatebox{90}{$\overline{(1111100)}$} &  \rotatebox{90}{$\overline{(1111110)}$} &  \rotatebox{90}{$\overline{(1111111)}$} &  \rotatebox{90}{$\overline{(1112100)}$} &  \rotatebox{90}{$\overline{(1112110)}$} &  \rotatebox{90}{$\overline{(1112111)}$} &  \rotatebox{90}{$\overline{(1112210)}$} &  \rotatebox{90}{$\overline{(1112211)}$} &  \rotatebox{90}{$\overline{(1112221)}$} &  \rotatebox{90}{$\overline{(1122100)}$} &  \rotatebox{90}{$\overline{(1122110)}$} &  \rotatebox{90}{$\overline{(1122111)}$} &  \rotatebox{90}{$\overline{(1122210)}$} &  \rotatebox{90}{$\overline{(1122211)}$} &  \rotatebox{90}{$\overline{(1122221)}$} &  \rotatebox{90}{$\overline{(1123210)}$} &  \rotatebox{90}{$\overline{(1123211)}$} &  \rotatebox{90}{$\overline{(1123221)}$} &  \rotatebox{90}{$\overline{(1123321)}$} &  \rotatebox{90}{$\overline{(1223210)}$} &  \rotatebox{90}{$\overline{(1223211)}$} &  \rotatebox{90}{$\overline{(1223221)}$} &  \rotatebox{90}{$\overline{(1223321)}$} &  \rotatebox{90}{$\overline{(1224321)}$} &  \rotatebox{90}{$\overline{(1234321)}$} &  \rotatebox{90}{$\overline{(2234321)}$}\\
		\hline
		\endfirsthead
		
		\hline
 &  \rotatebox{90}{$\overline{(0000001)}$} &  \rotatebox{90}{$\overline{(0000010)}$} &  \rotatebox{90}{$\overline{(0000011)}$} &  \rotatebox{90}{$\overline{(0000100)}$} &  \rotatebox{90}{$\overline{(0000110)}$} &  \rotatebox{90}{$\overline{(0000111)}$} &  \rotatebox{90}{$\overline{(0001000)}$} &  \rotatebox{90}{$\overline{(0001100)}$} &  \rotatebox{90}{$\overline{(0001110)}$} &  \rotatebox{90}{$\overline{(0001111)}$} &  \rotatebox{90}{$\overline{(0010000)}$} &  \rotatebox{90}{$\overline{(0011000)}$} &  \rotatebox{90}{$\overline{(0011100)}$} &  \rotatebox{90}{$\overline{(0011110)}$} &  \rotatebox{90}{$\overline{(0011111)}$} &  \rotatebox{90}{$\overline{(0100000)}$} &  \rotatebox{90}{$\overline{(0101000)}$} &  \rotatebox{90}{$\overline{(0101100)}$} &  \rotatebox{90}{$\overline{(0101110)}$} &  \rotatebox{90}{$\overline{(0101111)}$} &  \rotatebox{90}{$\overline{(0111000)}$} &  \rotatebox{90}{$\overline{(0111100)}$} &  \rotatebox{90}{$\overline{(0111110)}$} &  \rotatebox{90}{$\overline{(0111111)}$} &  \rotatebox{90}{$\overline{(0112100)}$} &  \rotatebox{90}{$\overline{(0112110)}$} &  \rotatebox{90}{$\overline{(0112111)}$} &  \rotatebox{90}{$\overline{(0112210)}$} &  \rotatebox{90}{$\overline{(0112211)}$} &  \rotatebox{90}{$\overline{(0112221)}$} &  \rotatebox{90}{$\overline{(1000000)}$} &  \rotatebox{90}{$\overline{(1010000)}$} &  \rotatebox{90}{$\overline{(1011000)}$} &  \rotatebox{90}{$\overline{(1011100)}$} &  \rotatebox{90}{$\overline{(1011110)}$} &  \rotatebox{90}{$\overline{(1011111)}$} &  \rotatebox{90}{$\overline{(1111000)}$} &  \rotatebox{90}{$\overline{(1111100)}$} &  \rotatebox{90}{$\overline{(1111110)}$} &  \rotatebox{90}{$\overline{(1111111)}$} &  \rotatebox{90}{$\overline{(1112100)}$} &  \rotatebox{90}{$\overline{(1112110)}$} &  \rotatebox{90}{$\overline{(1112111)}$} &  \rotatebox{90}{$\overline{(1112210)}$} &  \rotatebox{90}{$\overline{(1112211)}$} &  \rotatebox{90}{$\overline{(1112221)}$} &  \rotatebox{90}{$\overline{(1122100)}$} &  \rotatebox{90}{$\overline{(1122110)}$} &  \rotatebox{90}{$\overline{(1122111)}$} &  \rotatebox{90}{$\overline{(1122210)}$} &  \rotatebox{90}{$\overline{(1122211)}$} &  \rotatebox{90}{$\overline{(1122221)}$} &  \rotatebox{90}{$\overline{(1123210)}$} &  \rotatebox{90}{$\overline{(1123211)}$} &  \rotatebox{90}{$\overline{(1123221)}$} &  \rotatebox{90}{$\overline{(1123321)}$} &  \rotatebox{90}{$\overline{(1223210)}$} &  \rotatebox{90}{$\overline{(1223211)}$} &  \rotatebox{90}{$\overline{(1223221)}$} &  \rotatebox{90}{$\overline{(1223321)}$} &  \rotatebox{90}{$\overline{(1224321)}$} &  \rotatebox{90}{$\overline{(1234321)}$} &  \rotatebox{90}{$\overline{(2234321)}$}\\
		\hline
		\endhead
		
		\hline
		\endfoot
		\hline
		\endlastfoot

$\phi$ &  19 &  19 &  20 &  19 &  20 &  21 &  19 &  20 &  21 &  22 &  19 &  20 &  21 &  22 &  23 &  19 &  20 &  21 &  22 &  23 &  21 &  22 &  23 &  24 &  23 &  24 &  25 &  25 &  26 &  27 &  19 &  20 &  21 &  22 &  23 &  24 &  22 &  23 &  24 &  25 &  24 &  25 &  26 &  26 &  27 &  28 &  25 &  26 &  27 &  27 &  28 &  29 &  28 &  29 &  30 &  31 &  29 &  30 &  31 &  32 &  33 &  34 &  35\\
\hline
	\end{longtable}
	
\end{landscape}
\restoregeometry

	\clearpage
\newgeometry{left=0.2cm,right=0.2cm,top=1.2cm,bottom=0.2cm}		
\begin{landscape}
	\centering
	\fontsize{8.1pt}{1pt}\selectfont
	\setlength{\tabcolsep}{0pt}
	\renewcommand{\arraystretch}{0.6}
	
	\captionof{table}{Difference matrix $M_{E_8^{(1)}}$ (Part A)}
	\begin{longtable}{|c|*{82}{c|}}
\hline
 &  \rotatebox{90}{$\phi$} &  \rotatebox{90}{$(23465432)$} &  \rotatebox{90}{$(23465431)$} &  \rotatebox{90}{$(23465421)$} &  \rotatebox{90}{$(23465321)$} &  \rotatebox{90}{$(23464321)$} &  \rotatebox{90}{$(23454321)$} &  \rotatebox{90}{$(23354321)$} &  \rotatebox{90}{$(22454321)$} &  \rotatebox{90}{$(22354321)$} &  \rotatebox{90}{$(22344321)$} &  \rotatebox{90}{$(22343321)$} &  \rotatebox{90}{$(22343221)$} &  \rotatebox{90}{$(22343211)$} &  \rotatebox{90}{$(22343210)$} &  \rotatebox{90}{$(13354321)$} &  \rotatebox{90}{$(12354321)$} &  \rotatebox{90}{$(12344321)$} &  \rotatebox{90}{$(12343321)$} &  \rotatebox{90}{$(12343221)$} &  \rotatebox{90}{$(12343211)$} &  \rotatebox{90}{$(12343210)$} &  \rotatebox{90}{$(12244321)$} &  \rotatebox{90}{$(12243321)$} &  \rotatebox{90}{$(12243221)$} &  \rotatebox{90}{$(12243211)$} &  \rotatebox{90}{$(12243210)$} &  \rotatebox{90}{$(12233321)$} &  \rotatebox{90}{$(12233221)$} &  \rotatebox{90}{$(12233211)$} &  \rotatebox{90}{$(12233210)$} &  \rotatebox{90}{$(12232221)$} &  \rotatebox{90}{$(12232211)$} &  \rotatebox{90}{$(12232210)$} &  \rotatebox{90}{$(12232111)$} &  \rotatebox{90}{$(12232110)$} &  \rotatebox{90}{$(12232100)$} &  \rotatebox{90}{$(11233321)$} &  \rotatebox{90}{$(11233221)$} &  \rotatebox{90}{$(11233211)$} &  \rotatebox{90}{$(11233210)$} &  \rotatebox{90}{$(11232221)$} &  \rotatebox{90}{$(11232211)$} &  \rotatebox{90}{$(11232210)$} &  \rotatebox{90}{$(11232111)$} &  \rotatebox{90}{$(11232110)$} &  \rotatebox{90}{$(11232100)$} &  \rotatebox{90}{$(11222221)$} &  \rotatebox{90}{$(11222211)$} &  \rotatebox{90}{$(11222210)$} &  \rotatebox{90}{$(11222111)$} &  \rotatebox{90}{$(11222110)$} &  \rotatebox{90}{$(11222100)$} &  \rotatebox{90}{$(11221111)$} &  \rotatebox{90}{$(11221110)$} &  \rotatebox{90}{$(11221100)$} &  \rotatebox{90}{$(11221000)$} &  \rotatebox{90}{$(11122221)$} &  \rotatebox{90}{$(11122211)$} &  \rotatebox{90}{$(11122210)$} &  \rotatebox{90}{$(11122111)$} &  \rotatebox{90}{$(11122110)$} &  \rotatebox{90}{$(11122100)$} &  \rotatebox{90}{$(11121111)$} &  \rotatebox{90}{$(11121110)$} &  \rotatebox{90}{$(11121100)$} &  \rotatebox{90}{$(11121000)$} &  \rotatebox{90}{$(11111111)$} &  \rotatebox{90}{$(11111110)$} &  \rotatebox{90}{$(11111100)$} &  \rotatebox{90}{$(11111000)$} &  \rotatebox{90}{$(11110000)$} &  \rotatebox{90}{$(10111111)$} &  \rotatebox{90}{$(10111110)$} &  \rotatebox{90}{$(10111100)$} &  \rotatebox{90}{$(10111000)$} &  \rotatebox{90}{$(10110000)$} &  \rotatebox{90}{$(10100000)$} &  \rotatebox{90}{$(10000000)$} &  \rotatebox{90}{$(01122221)$} &  \rotatebox{90}{$(01122211)$} &  \rotatebox{90}{$(01122210)$}\\
\hline
\endfirsthead
\hline
 &  \rotatebox{90}{$\phi$} &  \rotatebox{90}{$(23465432)$} &  \rotatebox{90}{$(23465431)$} &  \rotatebox{90}{$(23465421)$} &  \rotatebox{90}{$(23465321)$} &  \rotatebox{90}{$(23464321)$} &  \rotatebox{90}{$(23454321)$} &  \rotatebox{90}{$(23354321)$} &  \rotatebox{90}{$(22454321)$} &  \rotatebox{90}{$(22354321)$} &  \rotatebox{90}{$(22344321)$} &  \rotatebox{90}{$(22343321)$} &  \rotatebox{90}{$(22343221)$} &  \rotatebox{90}{$(22343211)$} &  \rotatebox{90}{$(22343210)$} &  \rotatebox{90}{$(13354321)$} &  \rotatebox{90}{$(12354321)$} &  \rotatebox{90}{$(12344321)$} &  \rotatebox{90}{$(12343321)$} &  \rotatebox{90}{$(12343221)$} &  \rotatebox{90}{$(12343211)$} &  \rotatebox{90}{$(12343210)$} &  \rotatebox{90}{$(12244321)$} &  \rotatebox{90}{$(12243321)$} &  \rotatebox{90}{$(12243221)$} &  \rotatebox{90}{$(12243211)$} &  \rotatebox{90}{$(12243210)$} &  \rotatebox{90}{$(12233321)$} &  \rotatebox{90}{$(12233221)$} &  \rotatebox{90}{$(12233211)$} &  \rotatebox{90}{$(12233210)$} &  \rotatebox{90}{$(12232221)$} &  \rotatebox{90}{$(12232211)$} &  \rotatebox{90}{$(12232210)$} &  \rotatebox{90}{$(12232111)$} &  \rotatebox{90}{$(12232110)$} &  \rotatebox{90}{$(12232100)$} &  \rotatebox{90}{$(11233321)$} &  \rotatebox{90}{$(11233221)$} &  \rotatebox{90}{$(11233211)$} &  \rotatebox{90}{$(11233210)$} &  \rotatebox{90}{$(11232221)$} &  \rotatebox{90}{$(11232211)$} &  \rotatebox{90}{$(11232210)$} &  \rotatebox{90}{$(11232111)$} &  \rotatebox{90}{$(11232110)$} &  \rotatebox{90}{$(11232100)$} &  \rotatebox{90}{$(11222221)$} &  \rotatebox{90}{$(11222211)$} &  \rotatebox{90}{$(11222210)$} &  \rotatebox{90}{$(11222111)$} &  \rotatebox{90}{$(11222110)$} &  \rotatebox{90}{$(11222100)$} &  \rotatebox{90}{$(11221111)$} &  \rotatebox{90}{$(11221110)$} &  \rotatebox{90}{$(11221100)$} &  \rotatebox{90}{$(11221000)$} &  \rotatebox{90}{$(11122221)$} &  \rotatebox{90}{$(11122211)$} &  \rotatebox{90}{$(11122210)$} &  \rotatebox{90}{$(11122111)$} &  \rotatebox{90}{$(11122110)$} &  \rotatebox{90}{$(11122100)$} &  \rotatebox{90}{$(11121111)$} &  \rotatebox{90}{$(11121110)$} &  \rotatebox{90}{$(11121100)$} &  \rotatebox{90}{$(11121000)$} &  \rotatebox{90}{$(11111111)$} &  \rotatebox{90}{$(11111110)$} &  \rotatebox{90}{$(11111100)$} &  \rotatebox{90}{$(11111000)$} &  \rotatebox{90}{$(11110000)$} &  \rotatebox{90}{$(10111111)$} &  \rotatebox{90}{$(10111110)$} &  \rotatebox{90}{$(10111100)$} &  \rotatebox{90}{$(10111000)$} &  \rotatebox{90}{$(10110000)$} &  \rotatebox{90}{$(10100000)$} &  \rotatebox{90}{$(10000000)$} &  \rotatebox{90}{$(01122221)$} &  \rotatebox{90}{$(01122211)$} &  \rotatebox{90}{$(01122210)$}\\
\hline
\endhead
$\phi$ &  0 &  1 &  2 &  3 &  4 &  5 &  6 &  7 &  7 &  8 &  9 &  10 &  11 &  12 &  13 &  8 &  9 &  10 &  11 &  12 &  13 &  14 &  11 &  12 &  13 &  14 &  15 &  13 &  14 &  15 &  16 &  15 &  16 &  17 &  17 &  18 &  19 &  14 &  15 &  16 &  17 &  16 &  17 &  18 &  18 &  19 &  20 &  17 &  18 &  19 &  19 &  20 &  21 &  20 &  21 &  22 &  23 &  18 &  19 &  20 &  20 &  21 &  22 &  21 &  22 &  23 &  24 &  22 &  23 &  24 &  25 &  26 &  23 &  24 &  25 &  26 &  27 &  28 &  29 &  19 &  20 &  21\\
\hline
\end{longtable}	
\end{landscape}
\restoregeometry

\clearpage
\newgeometry{left=0.2cm,right=0.2cm,top=1.2cm,bottom=0.2cm}
\begin{landscape}
	\centering
	\fontsize{7.5pt}{4pt}\selectfont
	\setlength{\tabcolsep}{0pt}
	\renewcommand{\arraystretch}{0.6}
	
	\addtocounter{table}{-1}
	\captionsetup{labelformat=empty}
	\captionof{table}{Difference matrix $M_{E_8^{(1)}}$ (Part B, continued)}
	\captionsetup{labelformat=default}
	
	%
	
\end{landscape}
\restoregeometry

\clearpage
\newgeometry{left=0.2cm,right=0.2cm,top=1.2cm,bottom=0.2cm}
\begin{landscape}
	\centering
	\fontsize{6.5pt}{4pt}\selectfont
	\setlength{\tabcolsep}{0pt}
	\renewcommand{\arraystretch}{0.6}
	
	\addtocounter{table}{-1}
	\captionsetup{labelformat=empty}
	\captionof{table}{Difference matrix $M_{E_8^{(1)}}$ (Part C, continued)}
	\captionsetup{labelformat=default}
	
	%
	
\end{landscape}
\restoregeometry

\end{document}